\documentclass[12pt]{amsart}
\usepackage{enumitem,amssymb,version,aliascnt,geometry,mathrsfs,chngcntr}
\usepackage[all]{xy}
\usepackage{hyperref}

\hypersetup{
pdfauthor={Olivier Haution},
pdftitle={Fixed point theorems involving numerical invariants},
pdfkeywords={p-group actions, fixed points, numerical invariants, projective varieties},
hidelinks
}

\usepackage[T1]{fontenc}
\usepackage{textcomp}

\swapnumbers

\DeclareMathOperator{\id}{id}
\DeclareMathOperator{\Spec}{Spec}

\DeclareMathOperator{\im}{im}

\DeclareMathOperator{\CH}{CH}
\DeclareMathOperator{\carac}{char}
\DeclareMathOperator{\Hom}{Hom}
\DeclareMathOperator{\Aut}{Aut}
\DeclareMathOperator{\GL}{GL}
\DeclareMathOperator{\PGL}{PGL}
\DeclareMathOperator{\End}{End}
\DeclareMathOperator{\trdeg}{tr.deg.}
\DeclareMathOperator{\R}{R}

\newcommand{\colim}[1]{\mycolim_{#1}}
\DeclareMathOperator*{\mycolim}{colim}

\newcommand{\Tan}{T}
\newcommand{\Oc}{\mathcal{O}}
\newcommand{\Ec}{\mathcal{E}}
\newcommand{\Lc}{\mathcal{L}}
\newcommand{\Fc}{\mathcal{F}}
\newcommand{\Gc}{\mathcal{G}}
\newcommand{\Zz}{\mathbb{Z}}

\newcommand{\Zo}{\mathcal{Z}}
\newcommand{\Ic}{\mathcal{I}}

\newcommand{\Fp}{\mathbb{F}_p}
\newcommand{\Qq}{\mathbb{Q}}
\newcommand{\Pp}{\mathbb{P}}
\newcommand{\Aa}{\mathbb{A}}
\newcommand{\Gm}{\mathbb{G}_m}
\newcommand{\Ga}{\mathbb{G}_a}

\newcommand{\mup}{\mu_p}
\newcommand{\Laz}{\mathbb{L}}
\newcommand{\KG}[2]{K'_0({#1};{#2})}
\newcommand{\Char}[1]{\widehat{#1}}
\newcommand{\Mod}[2]{\mathcal{M}({#2};{#1})}

\newtheorem{theorem}{Theorem}[section]
\newaliascnt{proposition}{theorem}
\newtheorem{proposition}[proposition]{Proposition}
\newaliascnt{lemma}{theorem}
\newtheorem{lemma}[lemma]{Lemma}
\newaliascnt{corollary}{theorem}
\newtheorem{corollary}[corollary]{Corollary}

\theoremstyle{definition}
\newaliascnt{remark}{theorem}
\newtheorem{remark}[theorem]{Remark}
\newaliascnt{example}{theorem}
\newtheorem{example}[example]{Example}
\newaliascnt{definition}{theorem}
\newtheorem{definition}[definition]{Definition}
\newaliascnt{notation}{theorem}

\newtheorem{noname}[theorem]{}

\renewcommand{\theequation}{\thesection.\alph{equation}}
\newcommand{\rref}[1]{(\ref{#1})}
\newcommand{\dref}[2]{(\ref{#1}.\ref{#2})}

\begin{document}
\begin{abstract}
We exhibit invariants of smooth projective algebraic varieties with integer values, whose nonvanishing modulo $p$ prevents the existence of an action without fixed points of certain finite $p$-groups. The case of base fields of characteristic $p$ is included. Counterexamples are systematically provided to test the sharpness of our results.
\end{abstract}

\author{Olivier Haution}
\title{Fixed point theorems involving numerical invariants}
\email{olivier.haution at gmail.com}
\address{Mathematisches Institut, Ludwig-Maximilians-Universit\"at M\"unchen, Theresienstr.\ 39, D-80333 M\"unchen, Germany}
\thanks{This work was supported by the DFG grant HA 7702/1-1 and Heisenberg fellowship HA 7702/4-1.}

\maketitle

\section*{Introduction}
Consider a finite $p$-group $G$ acting acting on a smooth projective variety $X$ over an algebraically closed field. One of the first questions arising in this situation concerns the existence of fixed points, and finding effective methods to predict their existence is desirable. In this paper we exhibit invariants of the variety which may be used to detect fixed points. Each such invariant has the following key properties:
\begin{enumerate}[label=(\alph*)]
\item \label{req:indep} It depends solely on the intrinsic geometry of the variety $X$, and not on the particular group action.
 
\item \label{req:computation} It takes values in $\Zz$ and is accessible to computation.

\item \label{req:detect} If it is prime to $p$, then the group $G$ must fix a point of the variety $X$.
\end{enumerate}

It is instructive to start by looking at the case of finite sets, which are precisely zero-dimensional varieties (over a given algebraically closed field). They are classified by their cardinality, a numerical invariant which may be used to detect fixed points: a $p$-group acting on a finite set of cardinality prime to $p$ must fix a point. A direct generalisation of the cardinality of a finite set to higher dimensional (possibly singular and nonprojective) varieties is the so-called topological Euler characteristic, an integer $\chi(X)$ defined as the alternating sum of the $l$-adic Betti numbers of $X$ (with compact supports). Much like in the case of finite sets, this invariant satisfies the requirements \ref{req:indep} \ref{req:computation} \ref{req:detect}, provided that the characteristic of the base field differs from $p$.

In this paper, we restrict ourselves to the consideration of smooth projective varieties. This allows us to use intersection theory to produce more invariants satisfying \ref{req:indep} \ref{req:computation} \ref{req:detect}, and also to include the case of base fields of characteristic $p$. Perhaps the most natural way to construct a numerical invariant of a smooth projective variety using intersection theory is to take the degree of a product of Chern classes of its tangent bundle. Such invariants are called \emph{Chern numbers}; they satisfy the requirements \ref{req:indep} \ref{req:computation}. Our first result asserts that they also detect fixed points in the sense of \ref{req:detect}, under any of the following additional assumptions:
\begin{enumerate}[label=(\roman*)]
\item \label{charnum_ab} $G$ is abelian,
\item \label{charnum_charp} the characteristic of the base field is $p$,
\item \label{charnum_dim} $\dim X <p$.
\end{enumerate}

Another numerical invariant of a smooth projective variety $X$ is its arithmetic genus $\chi(X,\Oc_X)$. This is not quite a Chern number, but rather a $\Qq$-linear combination of Chern numbers which happens to take integral values on any smooth projective variety (this assertion is a form of the Hirzebruch-Riemann-Roch theorem). The arithmetic genus $\chi(X,\Oc_X)$ depends solely on the \emph{birational} geometry of $X$. In fact it contains all the birational information obtainable using Chern numbers: any linear combination of Chern numbers which is a birational invariant must be a multiple of the arithmetic genus \cite[Example~15.12.3 (c)]{Ful-In-98}. We prove that the invariant $\chi(X,\Oc_X)$ detects fixed points (in the sense of \ref{req:detect} above), under any of the following additional assumptions:
\begin{enumerate}[label=(\roman*')]
\item \label{chi_cycl} $G$ is cyclic,
\item \label{chi_charp} the characteristic of the base field is $p$,
\item \label{chi_dim} $\dim X <p-1$.
\end{enumerate}

Under any of the assumptions \ref{chi_cycl} \ref{chi_charp} \ref{chi_dim}, we prove in fact that the cobordism class of the variety $X$ is divisible by $p$ in the Lazard ring as soon as the group $G$ acts without fixed points on $X$. This essentially means that fixed points may be detected using any $\Qq$-linear combination of Chern numbers taking integral values on smooth projective varieties, and not just the arithmetic genus.\\

The results are stated in more details in \S\ref{sect:results}, where we also give an algebraic version of a theorem of Browder in topology \cite{Browder-PB} concerning the pull-back of fixed points under maps of degree prime to $p$. The rest of the paper is written using the language of algebraic groups. This allows us to prove more general statements (by including in particular infinitesimal groups in characteristic $p$), but also makes some arguments more transparent. For instance the conditions \ref{charnum_ab} and \ref{charnum_charp} appear as special cases of the condition that the $p$-group $G$ be trigonalisable.

A notable tool introduced in \S\ref{sect:KG} is a cohomology theory that we denote by $K_G$, which lies between the equivariant and usual $K_0$-theory. It is built using a variant of Borel's construction, where roughly speaking the space $EG$ is replaced by its generic point. The same construction may be applied to other cohomology theories (such as the Chow group), yielding a theory which keeps track of some of the equivariant information while being smaller than the Borel-type equivariant theory.\\

Even though these results are, in our view, principally of geometric nature, they also have an arithmetic component. Indeed our methods permit in some cases the detection of closed fixed points of degree prime to $p$, when the base field is not algebraically closed. This is so for Chern numbers (assuming that the base field contains enough roots of unity in case \ref{charnum_ab}), or the arithmetic genus in case \ref{chi_dim}. Strikingly, this does not work for the arithmetic genus in case \ref{chi_cycl}, nor in case \ref{chi_charp} over an imperfect base field (see \rref{ex:SB} and \rref{ex:zero_cycle} for counterexamples). This suggests a difference of nature between these situations, and may provide a reason for the absence of a uniform proof for the cases \ref{chi_cycl} and \ref{chi_dim}.

\paragraph{\textbf{Acknowledgements.}} I thank the anonymous referee for his/her careful reading and constructive comments.

\section{Statement of results}
\numberwithin{theorem}{subsection}
\numberwithin{lemma}{subsection}
\numberwithin{proposition}{subsection}
\numberwithin{corollary}{subsection}
\numberwithin{example}{subsection}
\numberwithin{notation}{subsection}
\numberwithin{definition}{subsection}
\numberwithin{remark}{subsection}
\counterwithin{equation}{subsection}
\renewcommand{\theequation}{\thesubsection.{\alph{equation}}}

\label{sect:results}

Here we summarise the main results of the paper. In order to make the statements as simple as possible, we assume in this section that $G$ is an ordinary finite $p$-group and that the base field $k$ is algebraically closed. A $G$-action on a ($k$-)variety $X$ is a group morphism $G \to \Aut_k(X)$. A coherent $\Oc_X$-module $\Fc$ is called $G$-equivariant if it is endowed with isomorphisms $\tau_g \colon g^*\Fc \to \Fc$ for $g \in G$, satisfying the cocycle condition $\tau_h \circ h^*(\tau_g) = \tau_{gh}$ for every $g,h \in G$.

\subsection{Equivariant cycles and Chern numbers}
\label{sect:Chern}
In the next statement $\CH_G$ denotes the equivariant Chow group of \cite{EG-Equ}.

\begin{theorem}
\label{th:constant_CH}
Let $X$ be a projective variety with an action of $G$. Assume that one of the following conditions holds:
\begin{enumerate}[label=(\roman*), ref=\roman*]
\item \label{th:constant_CH:ab} $G$ is abelian.
\item \label{th:constant_CH:p} $\carac k =p$.
\item \label{th:constant_CH:dim} $\dim X <p-1$.
\end{enumerate}
Then $X(k)^G = \varnothing$ if and only if every element of  the image of the morphism $\CH_G(X) \to \CH(X)$ has degree divisible by $p$.
\end{theorem}
\begin{proof}
The class of a fixed closed point belongs to $\im(\CH_G(X) \to \CH(X))$ and has degree one. The converse follows from \rref{th:eq_deg} in cases \eqref{th:constant_CH:ab} and \eqref{th:constant_CH:p}, since $G$ is trigonalisable (see \S\ref{ex:trig}), and from \dref{th:smalldim}{th:smalldim:CH} in case \eqref{th:constant_CH:dim}.
\end{proof}

A \emph{Chern number} of a smooth projective variety is the degree of a product of Chern classes (with values in the Chow group) of its tangent bundle.
\begin{corollary}
\label{cor:constant_CH}
Let $X$ be a smooth projective variety with an action of $G$. Assume that one of the following conditions holds:
\begin{enumerate}[label=(\roman*), ref=\roman*]
\item \label{cor:constant_CH:ab} $G$ is abelian.
\item \label{cor:constant_CH:p} $\carac k =p$.
\item \label{cor:constant_CH:dim} $\dim X <p$.
\end{enumerate}
If $X(k)^G = \varnothing$, then every Chern number of $X$ is divisible by $p$.
\end{corollary}
\begin{proof}
By \cite[\S2.4]{EG-Equ} any $G$-equivariant vector bundle over $X$ (i.e.\ a vector bundle whose $\Oc_X$-module of sections is $G$-equivariant) admits Chern classes with values in $\CH_G(X)$ mapping under the forgetful morphism $\CH_G(X) \to \CH(X)$ to its usual Chern classes. Since the tangent bundle of $X$ is $G$-equivariant (see \rref{rem:Tan_is_equ}), the result follows from \rref{th:constant_CH} in cases \eqref{cor:constant_CH:ab} and \eqref{cor:constant_CH:p} (or when $\dim X <p-1$). In case \eqref{cor:constant_CH:dim}, the result may also be deduced from \rref{th:constant_CH}, albeit less directly, see \dref{th:smalldim}{th:smalldim:Chern}.
\end{proof}

Counterexamples show that the conditions of the theorem, and of the corollary in case $p \neq 2$, are sharp, see \rref{ex:nonabelian} below. The dimensional bound \rref{cor:constant_CH:dim} in the corollary may however be improved when $p =2$, see \rref{rem:Euler_top:2}.

\begin{remark}
\label{rem:Euler_top} The following statement may be found in \cite[\S7.2]{Serre-finite_fields}: \emph{Assume that $\carac k \neq p$, and let $X$ be a variety with an action of $G$. If the topological Euler characteristic $\chi(X)$ is prime to $p$, then $X(k)^G \neq \varnothing$.}

The assumption on the characteristic of $k$ is necessary: the group $\Zz/p$ acts without fixed points on $\Aa^1$ by $x \mapsto x +1$ in characteristic $p$, and $\chi(\Aa^1)=1$. Note that when $X$ is smooth projective of pure dimension $d$, the integer $\chi(X) = \deg c_d(\Tan_X)$ is a Chern number of $X$. It thus follows from \dref{cor:constant_CH}{cor:constant_CH:p} that the above statement is in fact valid in characteristic $p$ when $X$ is smooth and projective.
\end{remark}

\begin{remark}
\label{rem:Euler_top:2}
Let us now comment on the situation in \rref{cor:constant_CH} when $p=2$ and $X$ has small pure dimension. We denote the Chern number $\deg c_{i_1}(\Tan_X) \cdots c_{i_n}(\Tan_X)$ by $c_{i_1} \cdots c_{i_n}$.

--- $\dim X=1$: the only Chern number $c_1 = \chi(X) = 2 \chi(X,\Oc_X)$ is even. Thus \rref{cor:constant_CH} says nothing for curves when $p=2$.

--- $\dim X =2$: the two Chern numbers $c_2 = \chi(X)$ and $c_1^2$ have the same parity by Noether's formula $\chi(X) + c_1^2 =12 \chi(X,\Oc_X)$ (see e.g.\ \cite[Example~15.2]{Ful-In-98}). Thus \rref{cor:constant_CH} reduces to the statement in \rref{rem:Euler_top} for surfaces when $p=2$.

--- $\dim X =3$: we have by the Hirzebruch-Riemann-Roch theorem (see e.g.\ \cite[Example~15.2.5 (a)]{Ful-In-98})
\[
c_1 c_2 = 24 \chi(X,\Oc_X) \quad \text{ and } \quad c_1 c_2 + 4 c_1^3 = 8 \chi(X,\omega_X^\vee),
\]
so that the two Chern numbers $c_1 c_2$ and $c_1^3$ are even. The other Chern number $c_3 = \chi(X)$ is also even (for instance by Poincar\'e duality for $l$-adic cohomology). Thus \rref{cor:constant_CH} says nothing for threefolds when $p=2$.

Summarising, we see that the condition \rref{cor:constant_CH:dim} of \rref{cor:constant_CH} may be weakened to the condition ``$\dim X <4$'' when $p=2$. This new dimensional bound is sharp, see \rref{ex:nonabelian}.
\end{remark}

\subsection{Equivariant modules and the cobordism ring}

\begin{theorem}
\label{th:constant_K}
Let $X$ be a projective variety with an action of $G$. Assume that one of the following conditions holds:
\begin{enumerate}[label=(\roman*),ref=\roman*]
\item \label{th:constant_K:cyclic} $G$ is cyclic.
\item \label{th:constant_K:carac} $\carac k =p$.
\item \label{th:constant_K:small} $\dim X <p-1$.
\end{enumerate}
Then $X(k)^G = \varnothing$ if and only if the Euler characteristic $\chi(X, \Fc)$ of every $G$-equivariant coherent $\Oc_X$-module $\Fc$ is divisible by $p$.
\end{theorem}
\begin{proof}
If $\Ic$ is the ideal in $\Oc_X$ of a fixed closed point, then the $\Oc_X$-module $\Oc_X/\Ic$ is $G$-equivariant and satisfies $\chi(X, \Oc_X/\Ic)=1$. The converse is proved respectively in \rref{th:cyclic}, \rref{th:K:carac}, and \dref{th:smalldim}{th:smalldim:K}.
\end{proof}

The sharpness of the conditions of the theorem is illustrated by \rref{ex:noncyclic} below.\\

Let us declare two smooth projective varieties equivalent if their collections of Chern numbers (indexed by the monomials in the Chern classes) coincide. The disjoint union operation endows the set of equivalence classes with the structure of a cancellative abelian monoid. Its group completion\footnote{Contrary to what is stated in \cite[just before Theorem 5.1]{ELW} and in \cite[\S5]{Thom-Bourbaki}, it is necessary to add an inverse to the classes of zero-dimensional varieties in order to obtain a group.} has a ring structure, where the product is induced by the cartesian product of varieties. This ring does not depend on the field $k$, in fact Merkurjev proved that it coincides with the Lazard ring $\Laz$, the coefficient ring of the universal commutative one-dimensional formal group law \cite[Theorem 8.2]{Mer-Ori}. Thus to each smooth projective variety $X$ corresponds a class $[X] \in \Laz$. When the characteristic of $k$ is zero, the ring $\Laz$ may be identified with the coefficient ring $\Omega(\Spec k)$ of the algebraic cobordism of Levine-Morel, and $[X]$ is the cobordism class of the morphism $X \to \Spec k$ \cite[Remark~4.3.4, Theorem~4.3.7]{LM-Al-07}.

\begin{corollary}
\label{cor:constant_K}
In the conditions of \rref{th:constant_K}, assume additionally that $X$ is smooth. If $X(k)^G = \varnothing$, then $[X] \in p\Laz$.
\end{corollary}
\begin{proof}
The sheaf of sections of the tangent bundle $\Tan_X$ is a $G$-equivariant coherent $\Oc_X$-module (see \rref{rem:Tan_is_equ}). The corollary thus follows from \rref{th:constant_K}, in view of the Hattori-Stong theorem, see \cite[Lemma~5.4, Proposition~5.5]{ELW}. In case \eqref{th:constant_K:small}, this is also a consequence of \rref{th:constant_CH}: indeed one may see using \cite[II, Theorem~7.8, Lemma 7.9~(iii)]{Adams-Stable} that the class of a smooth projective variety of dimension $<p-1$ is divisible by $p$ in $\Laz$ if and only if each of its Chern numbers is divisible by $p$.
\end{proof}

\begin{remark}[\rref{cor:constant_K} has a converse] If the group $G$ is nontrivial, then $p\Laz$ consists of classes of smooth projective varieties $X$ admitting a $G$-action such that $X(k)^G=\varnothing$. Indeed, we may find a subgroup $H\subset G$ of index $p$. Then $G$ acts without fixed points on the set $G/H$, whose cardinality is $p$. Thus, given any smooth projective variety $Y$, we may make $G$ act without fixed point on $X=Y^{\sqcup p}$ by permuting the copies of $Y$.
\end{remark}

\begin{remark}
Fixed points were known to exist under the assumption that $X$ is connected and $H^i(X,\Oc_X) = 0$ for $i>0$: in characteristic $p$ this follows from Smith theory and the Artin-Schreier sequence \cite[\S7.4, Remark]{Serre-finite_fields}; the cyclic case in characteristic $\neq p$ follows from the Lefschetz fixed point formula \cite[VI \S9]{FL-Ri-85}. Let us also mention a fixed point theorem of Koll\'ar for smooth projective separably rationally connected varieties \cite[Corollary~3.36]{Starr-arithmetic}.

Since the $\Oc_X$-module $\Oc_X$ admits a natural $G$-equivariant structure, \rref{th:constant_K} implies that a projective variety $X$ has a fixed point under the weaker assumption that the arithmetic genus $\chi(X,\Oc_X)$ is prime to $p$. By contrast to the above-mentioned results, this applies to singular varieties, and also to some smooth projective varieties which are not ``rational-like'', such as curves of genus two. 
\end{remark}

\subsection{Browder's theorem}
\label{sect:Browder}
The next statement is an analog of a theorem of Browder in topology \cite{Browder-PB}. 
\begin{theorem}
\label{th:Browder_surj}
Let $f\colon Y \to X$ be a projective surjective $G$-equivariant morphism, where $X$ is smooth and connected, and $Y$ is integral. Assume that $[k(Y):k(X)]$ is finite and prime to $p$, and that one of the following conditions holds:
\begin{enumerate}[label=(\roman*),ref=\roman*]
\item \label{th:Browder_ab} $G$ is abelian.
\item $\carac k =p$.
\item $\dim X < p$.
\end{enumerate}

Then the induced morphism $f^G \colon Y^G \to X^G$ is surjective.
\end{theorem}
\begin{proof}
See \rref{th:Browder} (and \dref{rem:Browder}{rem:Browder:surj} and \S\ref{ex:trig}). Observe that the last condition implies that the fibers of $f$ all have dimension $<p-1$, since $f$ is a dominant morphism between integral varieties of the same dimension.
\end{proof}

When $\carac k \neq p$, we prove in fact that the $\Zz_{(p)}$-module $H(X^G)$ is a direct summand of $H(Y^G)$, for any $\Zz_{(p)}$-linear cohomology theory $H$ satisfying the projection formula, for instance $H=\CH(-)\otimes \Zz_{(p)}$ (see \dref{rem:Browder}{rem:Browder:smooth}).

\begin{remark}
A different method is used by Koll\'ar and Szab\'o in \cite{Reichstein-Youssin} to treat questions of similar nature. In particular, the fact that $X(k)^G \neq \varnothing$ implies that $Y(k)^G \neq \varnothing$ under the condition \dref{th:Browder_surj}{th:Browder_ab} was already stated and proved in \cite[Proposition~A4]{Reichstein-Youssin}.
%
\end{remark}

\section{Group actions on varieties}
\label{sect:prelim}
A variety will be a separated scheme of finite type over $k$ with the property that every finite family of points is contained in some affine open subscheme. In particular quasi-projective schemes over $k$ will be considered varieties, and quotients of varieties by finite algebraic groups will remain varieties (see \rref{sect:quotient}). Products of varieties will be understood as cartesian products over $k$. Morphisms of varieties will be $k$-morphisms; unless otherwise stated a morphism will mean a morphism of varieties. If $X$ is a variety and $R$ a $k$-algebra, the set of $k$-morphisms $\Spec R \to X$ will be denoted by $X(R)$. The residue field at a point $x \in X$ will be denoted by $k(x)$.

\subsection{Algebraic groups}
We refer e.g.\ to \cite[Chapter 0, \S1]{GIT} for the basic definitions concerning algebraic groups and their actions. An algebraic group will mean an affine group scheme of finite type over $k$. When $Y$ and $X$ are two varieties with a $G$-action, we always will endow the product $X \times Y$ with the diagonal $G$-action.

\subsubsection{The order of a group} The order of an algebraic group $G$ is the dimension of the $k$-vector space $H^0(G,\Oc_G)$. An algebraic group is called a \emph{$p$-group} if its order is (finite and) a power of the prime $p$. When $G$ is finite, the order of a subgroup $H$ of $G$ divides the order of $G$; this follows from the fact that the quotient morphism $G \to G/H$ is flat and finite of degree equal to the order of $H$ (see \rref{prop:G_torsor}).

\subsubsection{Constant finite groups}
Let $\Gamma$ be an ordinary finite group. Then there exists an algebraic group $\Gamma_k$ such that $\Gamma_k(R)=\Gamma$ for any connected commutative $k$-algebra $R$. The order of $\Gamma_k$ coincides with the order of $\Gamma$. If $G$ is an algebraic group and $\Gamma$ is a subgroup of $G(k)$, then the constant group $\Gamma_k$ is a subgroup of $G$.

\subsubsection{Diagonalisable groups}(see \cite[I, \S4.4]{SGA3-1}) \label{sect:diag}
A diagonalisable group $G$ is entirely determined by its character group $\Char{G}$, an ordinary abelian group. The order of $G$ coincides with the order of $\Char{G}$.
 
When $n$ is a positive integer, the functor assigning to a commutative $k$-algebra $R$ the multiplicative group consisting of those $x \in R$ such that $x^n =1$ is represented by the diagonalisable group $\mu_n$ satisfying $\Char{\mu_n} = \Zz/n$. When $k$ contains a root of unity order $n$, the algebraic group $\mu_n$ is isomorphic to the constant group $(\Zz/n)_k$.

There is an anti-equivalence between the category of affine varieties with a $G$-action and that of commutative $\Char{G}$-graded $k$-algebras of finite type \cite[I, 4.7.3.1]{SGA3-1}. 

\subsubsection{Unipotent groups} An algebraic group $G$ is called \emph{unipotent} if every nonzero $G$-representation over $k$ admits a nonzero trivial subrepresentation. Constant finite $p$-groups over a field of characteristic $p$ are unipotent. 

When $k$ has positive characteristic $p$, the functor assigning to a commutative $k$-algebra $R$ the additive group consisting of those $x \in R$ such that $x^p =0$ is represented by a unipotent group $\alpha_p$, whose order is $p$.

\subsubsection{Trigonalisable groups}
\label{ex:trig}
An algebraic group $G$ is called \emph{trigonalisable} if every nonzero $G$-representation over $k$ admits a subrepresentation of codimension one. It is equivalent to require that $G$ contain a normal unipotent subgroup $U$ such that $G/U$ is diagonalisable \cite[IV, \S2, Proposition 3.4]{Demazure-Gabriel-1}.

A finite constant group $G=\Gamma_k$ is trigonalisable in each of the following cases:
\begin{itemize}
\item[---] The ordinary group $\Gamma$ is abelian, and $k$ contains a root of unity of order equal to the exponent of $\Gamma$ (the group $G$ is diagonalisable, see \cite[Lemma~3.6.1]{isol}).

\item[---] The field $k$ has characteristic $p>0$, and $G$ is a $p$-group (it is then unipotent).
\end{itemize}

\subsection{Fixed locus and quotient}
\begin{noname}
{\emph {Invariant subschemes.}}
Let $X$ be a variety with an action of an algebraic group $G$. A closed (resp.\ An open) subscheme $S$ of $X$ is called \emph{$G$-invariant} if the composite morphism $G \times S \to G \times X \to X$ factors through $S \to X$. Note that the scheme theoretic image of a $G$-equivariant morphism is a $G$-invariant closed subscheme.
\end{noname}

\begin{noname}
{\emph {Fixed locus.}}
By \cite[Proposition A.8.10 (1)]{CGP} there is a $G$-invariant closed subscheme $X^G$ of $X$ with trivial $G$-action, such that for any variety $T$ with trivial $G$-action, the set of $G$-equivariant morphisms $T\to X$ coincides with the set of morphisms $T \to X^G$ (as subsets of $\Hom(T,X)$). When $G=\Gamma_k$ for an ordinary finite group $\Gamma$, a $G$-action is precisely a group morphism $\Gamma \to \Aut_k(X)$, and $X^G(k) = X(k)^\Gamma$.
\end{noname}

\begin{noname}
{\emph{Quotient.}}
\label{sect:quotient}
When it exists, we denote by $X \to X/G$ the categorical quotient morphism, defined as the coequaliser of the two morphisms $G \times X \to X$ (the action and the second projection). This morphism always exists in the category of varieties when $G$ is a finite algebraic group (see \cite[V, Th\'eor\`eme~4.1(i)]{SGA3-1}). A $G$-action on an affine scheme $X=\Spec A$ is given by a $k$-algebra morphism $\mu \colon A \to A \otimes_k R$, where $R=H^0(G,\Oc_G)$. The set
\begin{equation}
\label{eq:AG}
A^G = \{x \in A | \mu(x) = x \otimes 1 \in A \otimes_k R)\}
\end{equation}
is a $k$-subalgebra of $A$, and the morphism $X \to X/G$ is given by the inclusion $A^G \to A$.
\end{noname}

\begin{noname}
When $H$ is a finite normal subgroup of $G$, the variety $X/H$ inherits a $G/H$-action making the quotient morphism $X \to X/H$ a $G$-equivariant morphism. 
\end{noname}

\begin{lemma}
\label{lemm:quotient_radicial}
Let $G$ be a finite algebraic group acting on a variety $X$. Then the morphism $X^G \to X/G$ is radicial.
\end{lemma}
\begin{proof}
We may assume that $X=\Spec A$ and that $k$ is infinite. Let $R$ be the coordinate ring of $G$. Denote by $m\colon R \to R \otimes R$ the co-multiplication morphism (tensor products are taken over $k$) and by $\mu \colon A \to A \otimes R$ be the co-action morphism. Then we have
\begin{equation}
\label{eq:action}
(\id_A \otimes m)\circ \mu = (\mu \otimes \id_R) \circ \mu
\end{equation}
When $C$ is a $k$-algebra, we denote by $N_C \colon C\otimes R \to C$ the norm, which maps $x\in C \otimes R$ to the determinant of the $C$-linear morphism $\alpha_x \colon C \otimes R \to C \otimes R$ given by multiplication with $x$. If $h \colon C \to C'$ is a $k$-algebra morphism, then
\begin{equation}
\label{eq:ext}
h \circ N_C = N_{C'} \circ (h \otimes \id_R).
\end{equation}

The morphism $G \times G\to G \times G$ given by $(g,h) \mapsto (g,g^{-1}h)$ is an automorphism and is compatible with the first projection. Therefore the corresponding ring endomorphism $\varphi$ of $R \otimes R$ is an $R$-algebra automorphism, for the algebra structure given by $\id_R \otimes 1 \colon R \to R\otimes R$. Thus for $x \in A \otimes R \otimes R$, we have $\alpha_{(\id_A \otimes \varphi)(x)} = (\id_A \otimes \varphi) \circ \alpha_x \circ (\id_A \otimes \varphi)^{-1}$, and taking determinants we obtain
\begin{equation}
\label{eq:norm_phi}
N_{A\otimes R} \circ (\id_A \otimes \varphi) = N_{A\otimes R}.
\end{equation}
In addition, the composite of the above automorphism $G \times G \to G \times G$ with the multiplication map of $G$ is the second projection, so that
\begin{equation}
\label{eq:phi_m}
\varphi \circ m = 1 \otimes \id_R.
\end{equation}

Since
\begin{align*}
(\id_A \otimes 1) \circ N_A \circ \mu 
&= N_{A\otimes R} \circ (\id_A \otimes 1 \otimes \id_R) \circ \mu & \text{ by \eqref{eq:ext}}\\ 
&= N_{A\otimes R} \circ (\id_A \otimes \varphi) \circ (\id_A \otimes m) \circ \mu & \text{ by \eqref{eq:phi_m}}\\
&= N_{A\otimes R} \circ (\id_A \otimes m) \circ \mu & \text{ by \eqref{eq:norm_phi}}\\
&= N_{A\otimes R} \circ (\mu \otimes \id_R) \circ \mu& \text{ by \eqref{eq:action}}\\
&= \mu \circ N_A \circ \mu & \text{ by \eqref{eq:ext}}
\end{align*}
the map $N_A \circ \mu\colon A \to A$ factors through a map $N \colon A \to A^G$ (see \eqref{eq:AG}).

When $K$ is a field containing $k$, any element of $X^G(K)$ is given by a $G$-equivariant morphism $\Spec K \to X$ for the trivial $G$-action on $K$. Let $f \colon A \to K$ be the corresponding $k$-algebra morphism. Then for any $x\in A$, by \eqref{eq:ext} we have 
\begin{equation}
\label{eq:norm}
f \circ N_A \circ \mu(x) = N_K\circ (f \otimes \id_R) \circ \mu(x) = N_K (f(x) \otimes 1) = f(x)^n,
\end{equation}
where $n = \dim_k R$ is the order of $G$.

Let now $f_1,f_2 \colon A \to K$ be the $k$-algebra morphisms corresponding to a pair of elements of $X^G(K)$ mapping to the same element in $(X/G)(K)$. Then the restriction of $f_i$ to $A^G$ does not depend on $i$. It follows that the map 
\[
f_i \circ N_A \circ \mu = f_i|_{A^G} \circ N
\]
does not depend on $i$. In view of \eqref{eq:norm}, we have in $K$, for any $x \in A$
\[
f_1(x)^n = f_2(x)^n.
\]

Writing $n=p^mq$ where $p$ is the characteristic exponent of $k$ and $q$ is prime to $p$, we obtain in $K$, for any $x\in A$,
\[
f_1(x)^q = f_2(x)^q.
\]
Let now $a\in A$ and $\lambda \in k$. The above equality for $x=a+\lambda$ yields
\[
(f_1(a)+\lambda)^q = (f_1(a + \lambda))^q = (f_2(a + \lambda))^q = (f_2(a)+\lambda)^q,
\]
so that any element of $k \subset K$ is a root of the polynomial
\[
(f_1(a)+X)^q - (f_2(a)+X)^q \in K[X],
\]
which must therefore vanish (as $k$ is infinite). Its coefficient at $X^{q-1}$ is
\[
q \cdot f_1(a) - q \cdot f_2(a).
\]
Since $q$ is invertible in $k$, we conclude that $f_1 = f_2$.
\end{proof}

\subsection{Free actions}

\begin{definition}
Let $G$ be an algebraic group acting on a variety $X$. The group $G$ \emph{acts freely on $X$} if the morphism $\gamma \colon G \times X \to X \times X$ given by $(g,x) \mapsto (g\cdot x,x)$ is a monomorphism. It is equivalent to require that for every variety $B$ and $x \in X(B)$ the set of those $g \in G(B)$ such that $g \cdot x =x$ be reduced to the unit element $1 \in G(B)$.
\end{definition}

\begin{proposition}
\label{prop:G_torsor}
Let $G$ be a finite algebraic group acting freely on a variety $X$. Then the quotient morphism $\varphi \colon X \to X/G$ is flat and finite of degree equal to the order of $G$, and the following square is cartesian, where the unlabelled arrows denote respectively the second projection and the action morphism
\[ \xymatrix{
G \times X\ar[r] \ar[d] & X \ar[d]^\varphi \\ 
X \ar[r]^\varphi & X/G
}\]
\end{proposition}
\begin{proof}
See \cite[\S 12, Theorem~1(B)]{Mumford-AV}, or \cite[V, Th\'eor\`eme~4.1(iv)]{SGA3-1}.
\end{proof}

\begin{lemma}
\label{lemm:free_cartesian}
Let $G$ be a finite algebraic group acting freely on a variety $X$. Let $X' \to X$ be a $G$-equivariant morphism. Then $G$ acts feely on $X'$, and the following square is cartesian
\[ \xymatrix{
X'\ar[r] \ar[d] & X \ar[d] \\ 
X'/G \ar[r] & X/G
}\]
\end{lemma}
\begin{proof}
The first statement follows from the definition of a free action in terms of functors of points. To prove that the square is cartesian, by descent and \rref{prop:G_torsor} we may assume that $X=G \times S$, where $G$ acts trivially on $S$. Let $S'=X'/G$. We need to prove that the $G$-equivariant $S'$-morphism $X' \to G \times S'$ is an isomorphism. To do so, by descent and \rref{prop:G_torsor} we may assume that $X'=G \times S'$. But any equivariant $S'$-endomorphism of an $S'$-group scheme coincides with multiplication by the image of the unit, hence must be an isomorphism.
\end{proof}

\begin{lemma}
\label{lemm:nofix_free}
Let $G$ be an algebraic group of prime order acting on a variety $X$. If $X^G= \varnothing$, then the action is free.
\end{lemma}
\begin{proof}
For a $k$-scheme $F$ and $x,y \in X(F)$, we denote by $S_{x,y}$ the fiber of the morphism $\gamma \colon G \times X \to X \times X$ over $(x,y) \colon F \to X \times X$. Let us write $G_F = G \times F$ and $X_F=X \times F$. For any $k$-scheme $E$, the set $S_{x,y}(E)$ may be identified with the subset of $G_F(E)$ consisting of those $g$ such that $g \cdot y =x \in X(E)$. In particular $S_{x,x}$ is a subgroup of the $F$-group scheme $G_F$ (the stabiliser of $x$), and the $F$-scheme $S_{x,y}$ is isomorphic to $S_{x,x}$ as soon as $S_{x,y}(F) \neq \varnothing$.

We now prove the statement. The morphism $\gamma \colon G \times X \to X \times X$ being proper, by \cite[(8.11.5)]{ega-4-3} it will suffice to prove that $S_{x,y} \to F$ is either empty or an isomorphism, when $F=\Spec K$ for an algebraically closed field $K$ containing $k$, and $x,y \in X(F)$. Assume that $S_{x,y}$ is nonempty. Then the finite type $K$-scheme $S_{x,y}$ possesses a rational point. It follows from the discussion above that the $F$-scheme $S_{x,y}$ is isomorphic to the subgroup $S_{x,x}$ of $G_F$. But the $K$-group scheme $G_F$ has exactly two subgroups (the order of a subgroup must divide its order). If $S_{x,x} = G_F$, then the morphism $x \colon F \to X$ is $G$-equivariant for the trivial action on $F$, hence factors through $X^G$; this would contradict the assumption that $X^G=\varnothing$. Thus the $K$-group scheme $S_{x,x}$ is trivial, and $S_{x,y} \to F$ is an isomorphism.
\end{proof}

\subsection{Equivariant modules}
\label{sect:equ-modules}

\begin{definition}
Let $G$ be an algebraic group acting on a variety $X$, and $\Fc$ a quasi-coherent $\Oc_X$-module. A \emph{$G$-equivariant structure} on $\Fc$ is an isomorphism $\mu \colon a^*\Fc \to p^*\Fc$, where  $a,p\colon G \times X \to X$ are respectively the action morphism and the second projection, making the following diagram commute
\begin{equation}
\label{diag:equivariant_mod}
\begin{aligned}
\xymatrix{
(\id_G \times a)^*a^*\Fc \ar[rr]^{(\id_G \times a)^*(\mu)} \ar@{=}[d]&& (\id_G \times a)^*p^*\Fc \ar@{=}[r]& q^*a^*\Fc \ar[d]^{q^*(\mu)} \\ 
m^*a^*\Fc \ar[rr]^{m^*(\mu)} &&m^*p^*\Fc \ar@{=}[r]& q^*p^*\Fc
}
\end{aligned}
\end{equation}
where $q \colon G \times G \times X \to G \times X$ is the projection on the last two factors, and $m\colon G \times G \times X\to G \times X$ is the base-change of the multiplication map of $G$. We will often omit to mention the isomorphism $\mu$, and simply say that $\Fc$ is a $G$-equivariant quasi-coherent $\Oc_X$-module. This is the same definition as \cite[I, \S6.5]{SGA3-1}, where the cocyle condition is instead expressed in terms of the inverse $\theta$ of $\mu$. We will say that a vector bundle is \emph{$G$-equivariant} when its sheaf of sections is a $G$-equivariant module.

Let $\Fc$ and $\Gc$ be two $G$-equivariant quasi-coherent $\Oc_X$-modules. A morphism $\phi \colon \Fc \to \Gc$ of $\Oc_X$-modules will be called $G$-equivariant if the following diagram commutes
\[ \xymatrix{
a^*\Fc\ar[r] \ar[d]_{a^*(\phi)} & p^*\Fc \ar[d]^{p^*(\phi)} \\ 
a^*\Gc \ar[r] & p^*\Gc
}\]
This defines a category $\Mod{G}{X}$. The category $\Mod{G}{\Spec k}$ (for the trivial $G$-action on $\Spec k$) is the category of $G$-representations over $k$. A $G$-equivariant morphism $f\colon Y \to X$ induces functors $f^* \colon \Mod{G}{X} \to \Mod{G}{Y}$ (see \cite[I, Lemme 6.3.1]{SGA3-1}) and $\R^if_* \colon \Mod{G}{Y} \to \Mod{G}{X}$ for every $i$ (the proof of \cite[I, Lemme 6.6.1]{SGA3-1} also works for higher direct images).
\end{definition}

\begin{noname}
A morphism of algebraic groups $H \to G$ induces a functor $\Mod{G}{X} \to \Mod{H}{X}$. The category $\Mod{1}{X}$ coincides with the category of quasi-coherent $\Oc_X$-modules.
\end{noname}

\begin{noname}
\label{action_is_equiv}
If $\Fc$ is a quasi-coherent $\Oc_X$-module, any $G$-equivariant structure $a^* \Fc \to p^* \Fc$ is a $G$-equivariant morphism. Thus we may define a morphism $a^* \to p^*$ of functors $\Mod{G}{G \times X} \to \Mod{G}{G \times X}$.
\end{noname}

\begin{noname}
If a diagonalisable group $G$ acts trivially on a variety $X$, and $\Fc$ is a quasi-coherent $\Oc_X$-module, a $G$-equivariant structure on $\Fc$ is the same thing as a $\Char{G}$-grading on the $\Oc_X$-module $\Fc$, i.e.\ a direct sum decomposition $\Fc = \bigoplus_{g \in \Char{G}} \Fc_g$ (see \cite[I, 4.7.3]{SGA3-1}).
\end{noname}

\begin{lemma}
\label{lemm:central_diag}
Let $G$ be an algebraic group acting on a variety $X$, and $D$ a central diagonalisable subgroup of $G$ acting trivially on $X$. Let $\Fc$ be a $G$-equivariant quasi-coherent $\Oc_X$-module. Then the $\Char{D}$-grading of the $\Oc_X$-module $\Fc$ lifts to the category $\Mod{G}{X}$.
\end{lemma}
\begin{proof}
We denote by $a,p \colon G \times X \to X$ the action and projection morphisms, and by $d \colon D \times X \to X$ their common restriction. Let $\alpha,\pi \colon G \times D \times X \to D \times X$ the base-changes of $a,p$, and $\delta \colon G \times D \times X \to G \times X$ the base-change of $d$. Let $m \colon G \times D \times X \to G \times X$ be the multiplication morphism. Let $s\colon G \times D \times X \to D \times G \times X$ the exchange morphism, and write $\alpha'=\alpha \circ s^{-1}$ and $\delta'=\delta \circ s^{-1}$. Since $D$ is central, the multiplication morphism $m' \colon D \times G \times X \to G \times X$ coincides with $ m \circ s^{-1}$. 

Denote by $\mu \colon a^*\Fc \to p^*\Fc$ the $G$-equivariant structure, and by $\tau \colon d^*\Fc \to d^*\Fc$ the induced $D$-equivariant structure. Using the commutativity of \eqref{diag:equivariant_mod} twice, we see that the following diagram commutes
\[ \xymatrix{
\delta^*a^*\Fc\ar@{=}[r] \ar[d]_{\delta^*(\mu)} & m^*a^*\Fc \ar[ddd]_{m^*(\mu)} \ar@{=}[r] & s^*m'^*a^*\Fc \ar@{=}[r] \ar[ddd]_{s^*m'^*(\mu)}& s^*\alpha'^*d^*\Fc \ar@{=}[r] \ar[d]_{s^*\alpha'^*(\tau)}& \alpha^*d^*\Fc \ar[d]^{\alpha^*(\tau)}\\ 
\delta^*p^*\Fc \ar@{=}[d] & & & s^*\alpha'^*d^*\Fc \ar@{=}[r]\ar@{=}[d]& \alpha^*d^*\Fc\ar@{=}[d]\\
\pi^*d^*\Fc \ar[d]_{\pi^*(\tau)} & & & s^*\delta'^*a^*\Fc \ar@{=}[r]\ar[d]_{s^*\delta'^*(\mu)}& \delta^*a^*\Fc\ar[d]^{\delta^*(\mu)}\\
\pi^*d^*\Fc \ar@{=}[r]& m^*p^*\Fc \ar@{=}[r] & s^*m'^*p^*\Fc \ar@{=}[r] & s^*\delta'^*p^*\Fc \ar@{=}[r] & \delta^*p^*\Fc
}\]

We let $D$ act on $G \times X$ via the trivial action on $G$. Then the morphisms $a,p \colon G \times X \to X$ are $D$-equivariant (since $D$ is central). The commutativity of the exterior rectangle asserts that the morphism $\mu$ is $D$-equivariant, and the lemma follows.
\end{proof}

\begin{definition}
\label{def:invariants}
Let $X$ be a variety with an action of $G$, and $H$ a normal subgroup of $G$ acting trivially on $X$. We let $G$ act on $H$ by $(g,h) \mapsto ghg^{-1}$. The second projection $\pi \colon H \times X \to X$ is $G$-equivariant, and coincides with the action morphism $\alpha \colon H \times X \to X$. The category $\Mod{G/H}{X}$ may be identified with the full subcategory of $\Mod{G}{X}$ consisting of those objects $\Fc$ on which $H$ acts trivially, i.e.\ such that $\pi^*\Fc = \alpha^*\Fc \to \pi^*\Fc$ is the identity. Consider the two morphisms $\id \to \pi_*\pi^*$ of functors $\Mod{G}{X} \to \Mod{G}{X}$, respectively adjoint to the morphism $\pi^* = \alpha^* \to \pi^*$ (see \rref{action_is_equiv}) and the identity of $\pi^*$. Their equaliser defines a morphism
\[
(-)^H \colon \Mod{G}{X} \to \Mod{G/H}{X}.
\]
\end{definition}

\begin{noname}
\label{rem:invariant_sections}
In the conditions of \rref{def:invariants}, let $\Fc \in \Mod{G}{X}$ and let $U$ be an open subscheme of $X$. Since taking the sections over $U$ is a left exact functor $\Mod{G}{X} \to \Mod{G}{\Spec k}$, it follows that $\Fc^H(U) = \Fc(U)^H$.
\end{noname}

\begin{noname}
\label{rem:Tan_is_equ}
The ideal of a $G$-invariant closed subscheme is naturally $G$-equivariant. Using this observation for the diagonal immersion $X \to X \times X$ and pulling back to $X$, we deduce that the cotangent module of $X$ is $G$-equivariant. Thus if $X$ is smooth, its tangent bundle is $G$-equivariant.
\end{noname}

\begin{proposition}
\label{prop:invariant_free}
Let $G$ be an algebraic group acting on a variety $X$. Let $H$ be a finite normal subgroup of $G$ acting freely on $X$, and denote by $\varphi \colon X \to X/H$ the quotient morphism. Then the composite
\[
\Mod{G/H}{X/H} \subset \Mod{G}{X/H} \xrightarrow{\varphi^*} \Mod{G}{X} 
\]
is an equivalence of categories, with quasi-inverse
\[
\Mod{G}{X} \xrightarrow{\varphi_*} \Mod{G}{X/H} \xrightarrow{(-)^H} \Mod{G/H}{X/H}.
\]
These equivalences preserve the full subcategories of coherent, resp.\ locally free coherent, equivariant modules.
\end{proposition}
\begin{proof}
Consider the morphism $\varphi^*((\varphi_*(-) )^H) \to \id$ of functors $\Mod{G}{X} \to \Mod{G}{X}$ which is adjoint to the natural morphism $(\varphi_*(-) )^H \to \varphi_*(-)$. It will suffice to prove that for every $G$-equivariant quasi-coherent $\Oc_X$-module $\Fc$, the morphism $\varphi^*((\varphi_*(\Fc) )^H) \to \Fc$ is an isomorphism of $\Oc_X$-modules (its inverse is then automatically $G$-equivariant). But this follows from descent theory: in view of \rref{prop:G_torsor}, the $H$-equivariant structure on $\Fc$ is precisely a descent datum with respect to the faithfully flat morphism $\varphi$.
\end{proof}

\section{Algebraic cycles}

\subsection{Grothendieck groups}
\begin{noname}
Let $X$ be a noetherian separated scheme. We denote by $K_0'(X)$ (resp.\ $K_0(X)$) the Grothendieck group (resp.\ ring) of coherent (resp.\ locally free coherent) $\Oc_X$-modules. The class of an $\Oc_X$-module $\Fc$ will be denoted by $[\Fc]$. The tensor product induces a $K_0(X)$-module structure on $K_0'(X)$. A proper, resp.\ flat, morphism $f\colon Y \to X$ induces a push-forward morphism $f_* \colon K_0'(Y) \to K_0'(X)$, resp.\ pull-back morphism $f^* \colon K_0'(X) \to K_0'(Y)$.
\end{noname}

\begin{noname}
A filtration is defined by letting $K_0'(X)_{(n)}$ be the subgroup of $K_0'(X)$ generated by images of the push-forward morphisms $K_0'(Z) \to K_0'(X)$, where $Z$ runs over the closed subschemes of $X$ of dimension $\leq n$. This filtration is compatible with push-forward morphisms \cite[X, Proposition 1.1.8]{sga6}. If $f \colon Y \to X$ is a flat morphism of relative dimension $d$, then $f^*K_0'(X)_{(n)} \subset K_0'(Y)_{(n+d)}$ (see \cite[X, \S3.3]{sga6}).
\end{noname}

\begin{noname}
When $X$ is a complete variety, we denote by $\chi(X,-) \colon K_0'(X) \to \Zz$ the push-forward morphism along the structural morphism $X \to \Spec k$. This is the unique morphism sending the class of a coherent $\Oc_X$-module $\Fc$ to its Euler characteristic
\begin{equation}
\label{def:Euler}
\chi(X,\Fc) = \sum_{i=0}^{\dim X} \dim_k H^i(X,\Fc).
\end{equation}
\end{noname}

\begin{noname}
Let now $X$ be a variety with an action of an algebraic group $G$. We denote by $K_0'(X;G)$ the Grothendieck group of $G$-equivariant coherent $\Oc_X$-modules \cite[\S1.4]{Thomason-group}.  Forgetting the $G$-equivariant structure induces a morphism $K_0'(X;G) \to K_0'(X)$. A finite, resp.\ flat, $G$-equivariant morphism $f\colon Y \to X$ induces a push-forward morphism $f_* \colon K_0'(Y;G) \to K_0'(X;G)$, resp.\ pull-back morphism $f^* \colon K_0'(X;G) \to K_0'(Y;G)$. The Grothendieck ring of $G$-equivariant locally free coherent $\Oc_X$-modules will be denoted by $K_0(X;G)$. Any $G$-equivariant morphism $f \colon Y \to X$ induces a pull-back morphism $K_0(X;G) \to K_0(Y;G)$, and $K_0'(X;G)$ is naturally a $K_0(X;G)$-module.
\end{noname}

\begin{proposition}
\label{prop:K_free}
Let $G$ be an algebraic group acting on $X$, and $H$ a finite normal subgroup of $G$ acting freely on $X$. Then the pull-back morphism $\KG{X/H}{G} \to \KG{X}{G}$ is surjective.
\end{proposition}
\begin{proof}
This follows from \rref{prop:invariant_free}.
\end{proof}

\subsection{Chow groups}

\begin{noname}
When $X$ is a variety, we denote by $\Zo(X)$ the free abelian group generated by the integral closed subschemes of $X$. Every closed subscheme $Z$ of $X$ has a class $[Z]\in \Zo(X)$ (defined as the sum of its irreducible components weighted by the lengths of the local rings of $Z$ at their generic points). The Chow group $\CH(X)$ is the quotient of $\Zo(X)$ by rational equivalence; we refer to \cite{Ful-In-98} (where the notation $A_*(X)$ is used) for its basic properties. 
\end{noname}

\begin{noname}
\label{sect:equiv_Chow}
When $X$ is a variety with an action of an algebraic group $G$, we will write $\CH_G(X)$ for the equivariant Chow group of $X$, denoted by $A^G_*(X)$ in \cite{EG-Equ}. This group is constructed using intersection theory on algebraic spaces. The consideration of algebraic spaces is not required when $G$ is a finite algebraic group (the main case of interest in this paper), since quotients of varieties always exist as varieties (see \S\ref{sect:quotient}).
\end{noname}

\begin{noname}
\label{sect:CHequ_loc}
Let $i \colon Y \to X$ be the immersion of a $G$-invariant closed subscheme, and $u\colon U \to X$ its open complement. The localisation sequence \cite[Proposition~1.8]{Ful-In-98} induces an exact sequence
\[
\CH_G(Y) \xrightarrow{i_*} \CH_G(X) \xrightarrow{u^*} \CH_G(U) \to 0.
\]
\end{noname}

\begin{noname}
The forgetful morphism $\CH_G(X) \to \CH(X)$ is defined as follows. Any element of $\CH_G(X)$ is represented by an element of $\CH( (X \times W)/G)$, where $W$ is a nonempty $G$-invariant open subscheme of a finite-dimensional $G$-representation where $G$ acts freely (the quotient $(X \times W)/G$ exists as an algebraic space by \cite[Corollary~6.3]{Artin-versal}). Composing the flat pull-back $\CH((X \times W)/G) \to \CH(X \times W)$ with the isomorphism $\CH(X \times W) \simeq \CH(X)$ of \rref{lemm:pt} below, we obtain a morphism $\CH_G(X) \to \CH(X)$.
\end{noname}

\begin{lemma}
\label{lemm:free_equpb}
Let $G$ be a finite algebraic group acting on $X$, and $H$ a normal subgroup of $G$ acting freely on $X$. Then the morphism $\CH_G(X/H) \to \CH_G(X)$ is surjective.
\end{lemma}
\begin{proof}
Let $Q=G/H$ and $Y=X/H$. Let $W$ be a $G$-invariant open subscheme of a finite dimensional $G$-representation where $G$ acts freely, and $U$ a $Q$-invariant open subscheme of a finite dimensional $Q$-representation where $Q$ acts freely. 

The morphism $f \colon Y \times U \times (W/H) \to Y \times U$ is flat and $Q$-equivariant. Since $Q$ acts freely on its target, it follows  by faithfully flat descent and \rref{lemm:free_cartesian} that the morphism $f/Q$, which may be identified with  $(Y \times U \times W)/G \to (Y \times U)/G$, is flat. 

The morphism $X \times U \times W \to Y \times U \times W$ is flat (by \rref{prop:G_torsor}) and $G$-equivariant. Since $G$ acts freely on its target, it follows as above that the morphism $(X \times U \times W)/G \times (Y \times U \times W)/G$ is flat.

Since $W$ is a $G$-invariant open subscheme of a finite dimensional $G$-representation, and $H$ acts freely on $X \times U$, it follows by faithfully flat descent, \rref{lemm:free_cartesian} and \rref{prop:invariant_free} that $(X \times U \times W)/H$ is a $Q$-invariant open subscheme of a $Q$-equivariant vector bundle over $(X \times U)/H = Y \times U$. Since $Q$ acts freely on $Y \times U$, by the same argument $((X \times U \times W)/H)/Q = (X \times U \times W)/G$ is an open subscheme of a vector bundle over $(Y \times U)/Q =(X \times U)/G$. 

Finally, the morphism $(X \times U)/G \to (Y\times U)/G$ is an isomorphism, hence we have a commutative diagram of flat pull-backs
\[ \xymatrix{
\CH((X \times U)/G) \ar[r]  & \CH((X \times U \times W)/G)\\ 
\CH((Y \times U)/G) \ar[r] \ar[u]^{\simeq} & \CH((Y \times U \times W)/G) \ar[u]
}\]
where the upper horizontal morphism is surjective by homotopy invariance and the localisation sequence \cite[Propositions~1.9 and 1.8]{Ful-In-98}. Therefore the right vertical morphism is surjective. Its component of degree $n$ coincides with that of $\CH_G(Y) \to \CH_G(X)$, as soon as $U$ and $W$ are chosen so that their closed complements both have codimension $>\dim X -n$.
\end{proof}

\begin{noname}
\label{sect:equ_refined}
We now explain how to lift the construction of \cite[Definition~8.1.1]{Ful-In-98} to equivariant Chow groups, as this will be needed in \S\ref{sect:browder}. Let $f \colon X \to Y$ be a $G$-equivariant morphism with $Y$ smooth. Let $Y' \to Y$ and $X' \to X$ be $G$-equivariant morphisms. Associated with the cartesian square of $G$-equivariant morphisms
\[ \xymatrix{
X' \times_Y Y'\ar[r] \ar[d] & X' \times Y' \ar[d] \\ 
X \ar[r]^{\gamma_f} & X \times Y
}\]
is a morphism $(\gamma_f)^!\colon \CH_G(X'\times Y') \to \CH_G(X'\times_Y Y')$ such that the diagram
\[ \xymatrix{
\CH_G(X'\times Y')\ar[rr]^{(\gamma_f)^!} \ar[d] && \CH_G(X'\times_Y Y') \ar[d] \\ 
\CH(X'\times Y') \ar[rr]^{(\gamma_f)^!} && \CH(X' \times_Y Y')
}\]
commutes (see \cite[\S3.1.1]{Bro-St-03}). Given $x \in \CH_G(X')$ and $y \in \CH_G(Y')$, let $x \times_G y \in \CH_G(X'\times Y')$ be their exterior product \cite[Definition-Proposition 2]{EG-Equ}, and define
\[
x\cdot_f y = (\gamma_f)^!(x \times_G y) \in \CH_G(X'\times_Y Y').
\]
\end{noname}

\subsection{Purely transcendental extensions}
Let us record here a couple of well-known observations:
\begin{lemma}
\label{lemm:pt}
Let $W$ be a nonempty open subscheme of the affine space $\Aa^d$, and let $X$ be a variety. Then the flat pull-backs $\CH(X) \to \CH(X \times W)$ and $K_0'(X)_{(n)} \to K_0'(X \times W)_{(n+d)}$ (for every integer $n$) are bijective.
\end{lemma}
\begin{proof}
We may find a nonconstant polynomial in $d$ variables with coefficients in $k$ whose nonvanishing locus $U$ is contained in $W$. Let $\delta$ be its degree. If the field $k$ is finite, it possesses two finite field extensions $k_0,k_1$ of coprime degree, each containing more than $\delta +1$ elements. By a restriction-corestriction argument, it will suffice to prove the statement after successively replacing $k$ with $k_0$ and $k_1$. Thus we may assume that $k$ contains $\delta +1$ distinct elements. Then one sees by induction on $m$ that any nonzero polynomial of degree $\leq \delta$ in $m$ variables with coefficients in $k$ induces a nonzero map $k^m \to k$. In particular $U$, and thus also $W$, possesses a rational point. Then by \cite[Lemma~55.7]{EKM}, resp. \cite[\S7.2.5]{Qui-72}, the pull-back along the corresponding regular closed immersion $i \colon X \to X \times W$ is a retraction of the flat pull-back $\CH(X) \to \CH(X \times W)$, resp.\ $K_0'(X) \to K_0'(X \times W)$. Since the latter is surjective by homotopy invariance and the localisation sequence (see \cite[Propositions~1.9 and 1.8]{Ful-In-98}, resp. \cite[Propositions~7.3.2 and 7.4.1]{Qui-72}), it is bijective. In addition, we have $i^*K_0'(X \times W)_{(n+d)} \subset K_0'(X)_{(n)}$ by \cite[Theorem~83]{Gil-K-05}, concluding the proof.
\end{proof}

\begin{lemma}
\label{lemm:pt_fil}
Let $K$ be a finitely generated purely transcendental field extension of $k$. Then for any variety $X$, the pull-back $K_0'(X) \to K_0'(X \times \Spec K)$ is an isomorphism of filtered groups.
\end{lemma}
\begin{proof}
Let $d = \trdeg(K/k)$. Then $E=\Spec K$ is the limit of a filtered family of nonempty open subschemes $W_\alpha$ of $\Aa^d$ with affine transition morphisms. Thus by \cite[Proposition~7.2.2]{Qui-72} the morphism 
\[
\colim{\alpha} K_0'(X \times W_\alpha) \to K_0'(X \times E)
\]
is bijective. Since each morphism $E \to W_\alpha$ is flat of relative dimension $-d$, the above morphism restricts for every $n$ to an injective morphism of subgroups
\begin{equation}
\label{eq:colim_fil}
\colim{\alpha} K_0'(X \times W_\alpha)_{(n+d)} \to K_0'(X \times E)_{(n)}.
\end{equation}
Let now $Z$ be a closed subscheme of $X \times E$ such that $\dim Z \leq n$. Consider the scheme theoretic closure $Z_\alpha$ of $Z$ in $X \times W_\alpha$. Any point of $x \in Z_\alpha$ is the specialisation of a point $y \in Z$, and 
\[
\trdeg(k(x)/k) \leq \trdeg(k(y)/k) = \trdeg(k(y)/K) + d \leq \dim Z +d.
\]
It follows that $\dim Z_\alpha \leq n+d$. We have a commutative diagram
\[ \xymatrix{
\colim{\alpha} K_0'(Z_\alpha) \ar[r] \ar[d] & K_0'(Z) \ar[d] \\ 
\colim{\alpha} K_0'(X \times W_\alpha) \ar[r] & K_0'(X \times E)}
\]
Since the upper horizontal morphism is surjective, we deduce that \eqref{eq:colim_fil} is surjective, and thus bijective. Now the statement follows from \rref{lemm:pt}.
\end{proof}

\subsection{An equivariant theory}
\label{sect:KG}
In this section, we introduce an equivariant theory $K_G$, for $G$ a finite algebraic group, which will be used exclusively in \S\ref{sect:small_dim}. Its key properties may be summarised as follows (when $X$ is a variety with a $G$-action):
\begin{enumerate}[label=(\roman*),ref=\roman*]
\item The forgetful morphism factors as $\KG{X}{G} \to K_G(X) \to K_0'(X)$. 
\item The morphism $\KG{X}{G} \to K_G(X)$ is surjective.
\item The morphism $K_G(X) \to K_0'(X)$ is bijective when $X=\Spec k$.
\item The group $K_G(X)$ is endowed with a filtration vanishing in negative degrees, which is compatible with the topological filtration of $K_0'(X)$ via the morphism $K_G(X) \to K_0'(X)$.
\end{enumerate}

\begin{remark}
The theory $K_G$ is constructed using the theory $K_0'$. One could perform the same construction using other cohomology theories instead; taking the Chow group would yield a quotient of the equivariant Chow group $\CH_G$.
\end{remark}

Let $G$ be a finite algebraic group. Let $V=H^0(G,\Oc_G)$ be the regular representation of $G$ over $k$. We view $E=\End(V)$ as a $G$-representation by letting $G$ act by left composition. Then the group $G$ acts freely on the nonempty $G$-invariant open subscheme $\GL(V)$ of $E$. The scheme
\[
T = \Spec  k(E)
\]
may be identified with the generic fiber of the morphism $E \to E/G$. Therefore $T$ is a variety over the field $F=k(E/G)=k(E)^G$, and is equipped with a free $G$-action.

\begin{definition}
When $X$ is a variety with a $G$-action, we may view the scheme $X \times T$ as an $F$-variety with a free $G$-action, and we define
\[
K_G(X) = K_0'( (X \times T)/G).
\]
A filtration is defined by letting $K_G(X)_{(n)} =  K_0'( (X \times T)/G)_{(n)}$.
\end{definition}

\begin{remark}
If $W$ is any generically free $G$-representation of finite dimension over $k$, replacing $T$ by $\Spec k(W)$ and $F$ by $k(W)^G$ in the above definition yields a canonically isomorphic filtered group. This observation will play no role in this paper, but may be used to construct a change of group morphism, for instance.
\end{remark}

We may use \rref{lemm:pt_fil} to define a morphism of filtered groups
\begin{equation}
\label{eq:K_G-K'}
K_G(X) = K_0'( (X \times T)/G) \to K_0'( X \times T) \simeq K_0'(X).
\end{equation}

Using the isomorphism induced by \rref{prop:invariant_free}, we construct a morphism
\begin{equation}
\label{eq:KG-K_G}
\KG{X}{G} \to \KG{X \times T}{G} \simeq K_0'( (X \times T)/G)=K_G(X),
\end{equation}
The composite of \eqref{eq:KG-K_G} and \eqref{eq:K_G-K'} is the canonical morphism $\KG{X}{G} \to K_0'(X)$.

\begin{lemma}
\label{lemm:KG-K_G}
The morphism \eqref{eq:KG-K_G} is surjective.
\end{lemma}
\begin{proof}
This morphism may be factored as
\begin{equation}
\label{eq:factor}
\KG{X}{G} \to \KG{X \times \GL(V)}{G}  \simeq K_0'( (X \times \GL(V))/G) \to K_0'( (X \times T)/G).
\end{equation}
The first morphism is surjective by the equivariant versions of homotopy invariance and of the localisation sequence \cite[Theorems 2.7 and 4.1]{Thomason-group}, the middle isomorphism follows from \rref{prop:invariant_free}. The scheme $\Spec F=T/G$ is a filtered limit of open subschemes of $\GL(V)/G$ with affine transition morphisms, hence by base-change (in view of \rref{lemm:free_cartesian}) the scheme $(X \times T)/G$ is a filtered limit of open subschemes of $(X \times \GL(V))/G$ with affine transition morphisms. Thus the last morphism of \eqref{eq:factor} is surjective by \cite[Propositions 7.2.2 and 7.3.2]{Qui-72}.
\end{proof}

A flat, resp.\ proper, $G$-equivariant morphism $f\colon Y \to X$ induces a morphism $f^* \colon K_G(X) \to K_G(Y)$, resp.\ $f_* \colon K_G(Y) \to K_G(X)$. The Grothendieck ring $K_0(X;G)$ naturally acts on $K_G(X)$: any element $\alpha \in K_0(X;G)$ pulls back to an element of $K_0(X \times T;G)$, which acts on $\KG{X \times T}{G} = K_G(X)$. We denote this action by $\alpha \cap - \colon K_G(X) \to K_G(X)$.

\begin{lemma}
\label{lemm:loc-K_G}
Let $i \colon Y \to X$ be the immersion of a $G$-invariant closed subscheme, and $u\colon U \to X$ its open complement. Then the following sequence is exact:
\[
K_G(Y) \xrightarrow{i_*} K_G(X) \xrightarrow{u^*} K_G(U) \to 0.
\]
\end{lemma}
\begin{proof}
Since $(Y \times T)/G$ is a closed subscheme of $(X \times T)/G$ with open complement $(U \times T)/G$ by \rref{lemm:free_cartesian} and faithfully flat descent, the lemma follows from the localisation sequence for $K_0'$ (see \cite[Proposition 7.3.2]{Qui-72}).
\end{proof}

\section{The degree of equivariant cycles}
\numberwithin{theorem}{section}
\numberwithin{lemma}{section}
\numberwithin{proposition}{section}
\numberwithin{corollary}{section}
\numberwithin{example}{section}
\numberwithin{notation}{section}
\numberwithin{definition}{section}
\numberwithin{remark}{section}
\counterwithin{equation}{section}
\numberwithin{equation}{section}
\renewcommand{\theequation}{\thesection.{\alph{equation}}}

\begin{definition}
Let $G$ be an algebraic group acting on a variety $X$. We denote by $\Zo_G(X)$ the subgroup of $\Zo(X)$ generated by classes of equidimensional $G$-invariant closed subschemes of $X$.
\end{definition}

If $U$ is a $G$-invariant open subscheme of $X$, then the restriction morphism $\Zo(X) \to \Zo(U)$ maps the subgroup $\Zo_G(X)$ surjectively onto $\Zo_G(U)$.

\begin{lemma}
\label{lemm:Cartier_rep}
Let $i\colon D \to X$ be a $G$-equivariant principal effective Cartier divisor. Then the image of the composite $\Zo_G(X) \to \CH(X) \xrightarrow{i^*} \CH(D)$ is contained in the image of the morphism $\Zo_G(D) \to \CH(D)$.
\end{lemma}
\begin{proof}
Let $Z$ be an equidimensional $G$-invariant closed subscheme of $X$. Let $Z_0$ (resp.\ $Z_1$) be the scheme theoretic closure of $Z -(D \cap Z)$ (resp.\ of $Z-Z_0$) in $Z$. Then $Z_0$ is an equidimensional $G$-invariant closed subscheme of $X$, and we have $[Z]= [Z_0] + [Z_1]$ in $\Zo(X)$. In addition the closed subscheme $Z_1$ is supported inside $D$, hence we may find $y\in \Zo(D)$ such $[Z_1]=i_*(y)$ in $\Zo(X)$. Since $i$ is a principal effective Cartier divisor, it follows from \cite[Proposition 2.6 (c)]{Ful-In-98} that the composite $i^* \circ i_*$ is the zero endomorphism of $\CH(D)$, hence we have in $\CH(D)$
\[
i^*[Z] = i^*[Z_0] + i^*[Z_1] = i^*[Z_0] + i^* \circ i_*(y) = i^*[Z_0].
\]
By construction, no associated point of $Z_0$ lies in $D$, so that $D \cap Z_0 \to Z_0$ is an effective Cartier divisor. Then by  \cite[Lemma~1.7.2]{Ful-In-98} the cycle class $i^*[Z]=i^*[Z_0] \in \CH(D)$ is represented by the cycle $[Z_0 \cap D] \in \Zo(D)$, which belongs to $\Zo_G(D)$, since $Z_0 \cap D$ is an equidimensional $G$-invariant closed subscheme of $D$.
\end{proof}

We recall that $\CH_G$ denotes the equivariant Chow group, see \S\ref{sect:equiv_Chow}.

\begin{proposition}
\label{prop:equiv_forget}
Let $G$ be a trigonalisable (\S\ref{ex:trig}) algebraic group acting on a variety $X$. Then the two morphisms
\[
\CH_G(X) \to \CH(X) \quad \text{ and } \quad \Zo_G(X) \to \CH(X)
\]
have the same image.
\end{proposition}
\begin{proof}
Let $W$ be a nonempty $G$-invariant open subscheme of a finite-dimensional $G$-representation $V$ over $k$, and assume that $G$ acts freely on $W$. Consider the commutative diagram (see \rref{lemm:pt} for the indicated isomorphisms)
\[ \xymatrix{
\CH( (X \times W)/G) \ar[r] & \CH(X \times W) & \CH(X\times V) \ar[l]_{\simeq} &\CH(X) \ar[l]_{\simeq} \\ 
\Zo( (X \times W)/G) \ar[r] \ar@{->>}[u]& \Zo_G(X \times W)\ar[u]& \Zo_G(X\times V) \ar[u] \ar@{->>}[l] &\Zo_G(X) \ar[l] \ar[u]
}\]
(If $G$ is not finite, then $(X \times W)/G$ may be only an algebraic space.) To prove the statement, it will suffice to prove that for every finite-dimensional $G$-representation $V$ over $k$, the two morphisms
\[
\Zo_G(X \times V) \to \CH(X \times V) \text{ and } \Zo_G(X) \to \CH(X \times V)
\]
have the same image. We proceed by induction on $\dim V$, the case $V=0$ being clear. Assume that $V \neq 0$. Since $G$ is trigonalisable, we can find a subrepresentation $V' \subset V$ of codimension one. Let $i\colon X \times V' \to X \times V$ be the induced $G$-equivariant closed immersion. By \rref{lemm:Cartier_rep}, for any $x\in \Zo_G(X \times V)$, the cycle class $i^*(x) \in \CH(X \times V')$ is represented by a cycle in $\Zo_G(X \times V')$.

By induction, we may find a cycle $y \in \Zo_G(X)$ such that $i^*(x) = y \times [V'] \in \CH(X \times V')$. Using the surjectivity of $\CH(X) \to \CH(X \times V)$, we find $z \in \CH(X)$ such that $x = z \times [V] \in \CH(X \times V)$. Then $z \times [V'] = i^*(x) = y \times [V'] \in \CH(X \times V')$ (see \cite[Lemma~55.7]{EKM}), hence $z = y \in \CH(X)$ by injectivity of $\CH(X) \to \CH(X \times V')$, and finally $x = y \times [V] \in \CH(X \times V)$, as required.
\end{proof}

\begin{theorem}
\label{th:eq_deg}
Let $G$ be a trigonalisable (\S\ref{ex:trig}) algebraic $p$-group acting on a projective variety $X$. Assume that the degree of every closed point of $X^G$ is divisible by $p$. Then the morphism $d_X \colon \CH_G(X) \to \CH(X) \xrightarrow{\deg} \Fp$ is zero.
\end{theorem}
\begin{proof}
By \rref{prop:equiv_forget}, it will suffice to prove that $\deg [T] =0 \in \Fp$ for any equidimensional $G$-invariant closed subscheme $T$ of $X$. It will of course suffice to consider the case $\dim T=0$. Replacing $X$ with $T$, we may assume that $\dim X=0$. Given a $G$-invariant closed subscheme $Z$ of $X$, consider the exact sequence of \S\ref{sect:CHequ_loc}
\[
\CH_G(Z) \xrightarrow{i_*} \CH_G(X) \xrightarrow{u^*} \CH_G(X - Z) \to 0,
\]
where $i \colon Z \to X$, resp.\ $u \colon X-Z \to X$, is the closed, resp.\ open,  immersion. Since $\dim X =0$, the open immersion $u$ is also closed, and the push-forward morphism $u_*$ is a section of $u^*$. In addition $d_X \circ i_* = d_Z$ and $d_X \circ u_* = d_{X -Z}$, so that
\begin{equation}
\label{eq:dx}
\im d_X = \im d_Z + \im d_{X-Z} \subset \Fp.
\end{equation}

Taking $Z=X^G$, we see that in order to prove the proposition, it will suffice to assume that $X^G = \varnothing$ and prove that $d_X = 0$. To do so, we may assume that $k$ is algebraically closed. Using \eqref{eq:dx} and noetherian induction, we may assume that $X$ possesses exactly two $G$-invariant closed subschemes ($\varnothing$ and $X$). It will thus suffice to prove that $\deg [X] =0 \in \Fp$. The choice of a rational point of $X$ induces a $G$-equivariant morphism $G \to X$, which must be scheme theoretically dominant. By \cite[V, Th\'eor\`eme 10.1.2]{SGA3-1}, there is a subgroup $H$ of $G$ such that $X$ is isomorphic to $G/H$. The morphism $G \to G/H$ is flat and finite by \rref{prop:G_torsor}, so that the integer $n=\dim_k H^0(X,\Oc_X)$ divides the order of $G$, and is thus a $p$-th power. Since $H \neq G$ (as $X^G = \varnothing$ and $X \neq \varnothing$), it follows that $\deg [X] = n\mod p$ vanishes in $\Fp$.
\end{proof}

\begin{example}
\label{ex:nonabelian}
Assume that the field $k$ is algebraically closed of characteristic $\neq p$, and let $G$ be a constant nonabelian finite $p$-group. Then $G$ admits a simple representation $V$ of dimension $d=p^n$ with $n >0$. Observe that the case $n=1$ actually occurs, when $G$ has order $p^3$. 

The group $G$ acts without fixed point on $\Pp(V)$, but the cycle class $c_1(\Oc(1))^{d-1}$ is in the image of $\CH_G(\Pp(V)) \to \CH(\Pp(V))$ and has degree prime to $p$ (equal to $1$). When $n=1$, we have $\dim \Pp(V) = p-1$.

The group $G$ also acts on $\Pp(V \oplus 1)$ with a single fixed point $\Pp(1)$. Let $f\colon X \to \Pp(V \oplus 1)$ be the blow-up at this point. There is a $G$-action on $X$ making $f$ a $G$-equivariant morphism. Then any fixed point of $X$ must be contained in the exceptional divisor $E$, which is $G$-equivariantly isomorphic to $\Pp(V)$. Since the latter has no fixed point, it follows that $X(k)^G=\varnothing$. The normal bundle of the effective Cartier divisor $j \colon \Pp(V) \simeq E \to X$ is $\Oc_{\Pp(V)}(-1)$, hence the zero-cycle class
\[
[E]^d = j_* \circ c_1(\Oc_{\Pp(V)}(-1))^{d-1}[\Pp(V)] \in \CH(X)
\]
has degree $(-1)^{d-1}$. The zero-cycle class on $\Pp(V \oplus 1)$
\[
\big(c_1(\Tan_{\Pp(V \oplus 1)})\big)^d = \big((d+1) \cdot c_1(\Oc_{\Pp(V \oplus 1)}(1))\big)^d \in \CH(\Pp(V \oplus 1))
\]
has degree $(d+1)^d$. By \cite[15.4.3]{Ful-In-98}, we have
\[
c_1(\Tan_X) = f^*c_1(\Tan_{\Pp(V \oplus 1)}) + (1-d)[E] \in \CH^1(X),
\]
so that $\deg (c_1(\Tan_X)^d) = 2 \mod p$. Thus we have found a Chern number of $X$ which is not divisible by $p$, when $p \neq 2$. When $n=1$, we have $\dim X = p$.

Now assume that $p=2$. The group $G$ acts on $P=\Pp(V \oplus 1) \times \Pp(V \oplus 1)$ with a single fixed point $\Pp(1) \times \Pp(1)$. Let $g \colon Y \to P$ be the blow-up at this point. As above, we have $Y(k)^G = \varnothing$, and by \cite[Example~15.4.2 (c)]{Ful-In-98}
\[
c_2(\Tan_Y) = g^*c_2(\Tan_P) \in  \CH^2(Y)/2,
\]
so that we have in $\mathbb{F}_2$
\[
\deg (c_2(\Tan_Y)^d) = \deg (c_2(\Tan_P)^d) = (\deg (c_1(\Tan_{\Pp(V \oplus 1)}))^d)^2 =1 \mod 2.
\]
Therefore one of the Chern numbers of $Y$ is odd. When $n=1$, we have $\dim Y =4$.
\end{example}

\section{The Euler characteristic of equivariant modules}
\numberwithin{theorem}{subsection}
\numberwithin{lemma}{subsection}
\numberwithin{proposition}{subsection}
\numberwithin{corollary}{subsection}
\numberwithin{example}{subsection}
\numberwithin{notation}{subsection}
\numberwithin{definition}{subsection}
\numberwithin{remark}{subsection}
\counterwithin{equation}{subsection}
\numberwithin{equation}{subsection}
\renewcommand{\theequation}{\thesubsection.{\alph{equation}}}

We now investigate divisibility properties of Euler characteristics of equivariant modules on a variety upon which a finite $p$-group acts without fixed points. The next proposition shows that the case of a free action can be easily handled. 
\begin{proposition}
\label{prop:free}
Let $G$ be a finite algebraic group acting freely on a projective variety $X$. Then 
the Euler characteristic $\chi(X,\Fc)$ of any $G$-equivariant coherent $\Oc_X$-module $\Fc$ is divisible by the order of $G$.
\end{proposition}
\begin{proof}
By \rref{prop:invariant_free} we may find a coherent $\Oc_{X/G}$-module $\Gc$ whose pull-back to $X$ is isomorphic to $\Fc$. Then $\chi(X,\Fc) = n \cdot \chi(X/G,\Gc)$, where $n$ is the order of $G$, see \cite[\S12, Theorem 2]{Mumford-AV}.
\end{proof}

\subsection{Dual cyclic groups}
\begin{theorem}
\label{th:cyclic}
Let $p$ be a prime, and $G=\mu_{p^r}$ for some integer $r$. Let $X$ be a projective variety with a $G$-action such that $X^G =\varnothing$. Then the Euler characteristic $\chi(X,\Fc)$ of any $G$-equivariant coherent $\Oc_X$-module $\Fc$ is divisible by $p$.
\end{theorem}
\begin{proof}
We may assume that $X \neq \varnothing$. Then $r \geq 1$, and $G$ contains a subgroup $H$ isomorphic to $\mu_{p^{r-1}}$. It follows from \rref{prop:fix_surj} below that $(X/H)^{G/H} = (X/H)^G = \varnothing$. Let $\psi \colon X \to X/H$ be the quotient morphism. By \rref{lemm:reduce_group} below, there is a $G/H$-equivariant coherent $\Oc_{X/H}$-module $\Gc$ such that $\chi(X/H,\Gc)$ is congruent modulo $p$ to $\chi(X/H,\psi_*\Fc) = \chi(X,\Fc)$. Replacing the variety $X$ with $X/H$, the group $G$ with $G/H$, and the module $\Fc$ with $\Gc$, we may assume that $G = \mu_p$. Then $G$ acts freely on $X$ by \rref{lemm:nofix_free}, and the theorem follows from \rref{prop:free}.
\end{proof}

\begin{lemma}
\label{lemm:reduce_group}
Let $G$ be an algebraic group acting on a projective variety $X$. Let $D$ be a diagonalisable central subgroup of $G$ acting trivially on $X$. Then for any prime $p$ and any $G$-equivariant coherent $\Oc_X$-module $\Fc$, we may find a $G/D$-equivariant coherent $\Oc_X$-module $\Gc$ such that
\[
\chi(X,\Fc) = \chi(X,\Gc) \mod p.
\]
\end{lemma}
\begin{proof}
By \rref{lemm:central_diag} we may assume that the $\Char{D}$-grading of $\Fc$ is concentrated in a single degree $d \in \Char{D}$. If $\chi(X,\Fc)$ is divisible by $p$, we may simply take $\Gc=0$. Otherwise by the definition \eqref{def:Euler} there is an integer $i$ such that $\dim_k H^i(X,\Fc)$ is prime to $p$. Then $V= H^i(X,\Fc)$ is a $G$-representation whose induced $\Char{D}$-grading is concentrated in degree $d$. The group $D$ acts trivially on the $G$-equivariant coherent $\Oc_X$-module $\Fc \otimes_k V^\vee$, and the integer $\chi(X,\Fc \otimes_k V^\vee) = \chi(X,\Fc) \cdot \dim_k V^\vee$ is prime to $p$. We may thus construct $\Gc$ as a direct sum of copies of $\Fc \otimes_k V^\vee$.
\end{proof}

\begin{proposition}
\label{prop:fix_surj}
Let $p$ be a prime, and $G=\mu_{p^r}$ for some integer $r$. Let $X$ be a variety with a $G$-action and $H$ a subgroup of $G$ such that $H \neq G$. Then the morphism $X^G \to (X/H)^G$ is surjective.
\end{proposition}
\begin{proof}
Let us write $Y=X/H$. Let $y \in Y^G(K)$ for some algebraically closed field extension $K$ of $k$ with trivial $G$-action. We may lift $y \in Y(K)$ to a point $x \in X(K)$, and it will suffice to prove that $x \in X^G(K)$. To do so, we may assume that $X=\Spec A$. Then we have a $k$-algebra morphism $f \colon A \to K$ whose restriction $A^H \to K$ (see \eqref{eq:AG}) induces a $G$-equivariant morphism (of schemes), and want prove that $f$ induces a $G$-equivariant morphism.

A $G$-action on $\Spec A$ is a $k$-vector space decomposition $A = \bigoplus_{i \in \Zz/p^r} A_i$ such that $A_i \cdot A_j \subset A_{i+j}$ (see \S\ref{sect:diag}). The group $H^* = \ker (\Zz/p^r \to \Char{H})$ is nontrivial by the assumption $H \neq G$, and we have $A^H = \bigoplus_{j \in H^*} A_j$. The assumption that $f|_{A^H}$ induces a $G$-equivariant morphism means that $f|_{A_j} =0$ whenever $j \in H^* -\{0\}$. Let now $i \in \Zz/p^r - \{0\}$ and $a\in A_i$. Since any two nontrivial subgroups of a cyclic $p$-group must have a nontrivial intersection, we may find an integer $n$ such that $ni \in H^* - \{0\}$. Then $a^n \in A_{ni} $ with $ni \in H^*-\{0\}$, hence $0=f(a^n) = f(a)^n$. Since $K$ is a field, we conclude that $f(a)=0$, proving that $f$ induces a $G$-equivariant morphism.
\end{proof}

\begin{example}
\label{ex:mup_zp}
We define a group morphism $\mup \to \GL_p$ by mapping $\xi \in \mup(R)$ for a $k$-algebra $R$ to the diagonal matrix with coefficients $1,\xi,\cdots,\xi^{p-1}$. This induces an action of $\mup$ on $X=\Pp^{p-1}$. If $K$ is a field extension of $k$, then $X^{\mup}(K)$ consists of the $p$ points $[1: 0: \cdots:0], [0:1:0\cdots:0],\cdots,[0:\cdots:0:1]$.

We define a group morphism $(\Zz/p)_k \to \GL_p$ by mapping the generator to the automorphism $(x_0,\cdots,x_{p-1}) \mapsto (x_1,\cdots,x_{p-1},x_0)$. This induces an action of $(\Zz/p)_k$ on $X$. If $K$ is a field extension of $k$, then $X^{(\Zz/p)_k}(K)=\{[1:\xi:\cdots:\xi^{p-1}]|\xi \in \mup(K)\}$.

The two actions commute, inducing an action of $G=\mup \times (\Zz/p)_k$ on $X$ such that $X^G = \varnothing$, while $\chi(X,\Oc_X)=1$ is prime to $p$.
\end{example}

\begin{example}
\label{ex:noncyclic}
Assume that the field $k$ is algebraically closed of characteristic $\neq p$. Let $G$ be a constant abelian noncyclic $p$-group. Then $G$ admits a quotient isomorphic to $(\Zz/p)_k \times (\Zz/p)_k$ and thus to $\mup \times (\Zz/p)_k$ by our assumptions on $k$. Therefore by \rref{ex:mup_zp} the group $G$ acts without fixed point on $\Pp^{p-1}$, and $\chi(\Pp^{p-1},\Oc_{\Pp^{p-1}})$ is prime to $p$.
\end{example}

The next two examples show that in \rref{th:cyclic} the assumption ``$X^G =\varnothing$'' may not be weakened to the assumption ``every closed point of $X^G$ has degree divisible by $p$''.

\begin{example}
\label{ex:SB} Let $X$ be the Severi-Brauer variety of a finite-dimensional central division $k$-algebra of degree $p$. Let a $p$-group $G$ act trivially on $X$. Then $\dim X = p-1$ and the degree of every closed point of $X^G$ is divisible by $p$, but $\chi(X,\Oc_X)=1$.\end{example}

\begin{example} \label{ex:zero_cycle} 
For a less degenerate example, assume that an element $a \in k-k^p$ exists. The automorphism $(x_0,\cdots,x_{p-1}) \mapsto (x_1,\cdots,x_{p-1},ax_0)$ in $\GL_p(k)$ maps to an element of order $p$ in $\PGL_p(k)$. This defines an action of $(\Zz/p)_k$ on $X=\Pp^{p-1}$. The fixed locus $X^{(\Zz/p)_k}$ is supported on the closed point $\{[1:z:\cdots:z^{p-1}]|z^p=a\}$ of degree $p$. Thus $\chi(X,\Oc_X)=1$ is not the degree of a zero-cycle supported on $X^{(\Zz/p)_k}$.
\end{example}

\subsection{Varieties of small dimension}
\label{sect:small_dim}
Let us consider the morphism (see \S\ref{sect:KG})
\[
\delta \colon K_G(X) \xrightarrow{\eqref{eq:K_G-K'}} K_0'(X) \xrightarrow{\chi(X,-)} \Zz \to \Fp.
\]

\begin{proposition}
\label{prop:delta_fixed}
Let $G$ be a finite algebraic group acting on a projective variety $X$. Assume that $G$ contains $\mup$ as a central subgroup. If $\dim X <p-1$, then
\[
\delta\big(K_G(X)\big) = \delta\big(K_G(X^{\mup})\big) \subset \Fp.
\]
\end{proposition}
\begin{proof}
The action of $\mup$ on $U=X-X^{\mup}$ is free by \rref{lemm:nofix_free}. Let $\phi \colon U \to U/\mup$ be the quotient morphism. The $\Zz/p$-grading $\phi_*\Oc_U = \bigoplus_{n \in \Zz/p} (\phi_*\Oc_U)_n$ is compatible with the $G$-equivariant structure of $\phi_*\Oc_U$ by \rref{lemm:central_diag}. By \cite[VIII, Proposition 4.1]{SGA3-1} the $G$-equivariant $\Oc_{U/\mup}$-module $\Lc=(\phi_*\Oc_U)_1$ is invertible, and for $i\in \Zz$, the morphism $\Lc^{\otimes i} \to (\phi_*\Oc_U)_{i \mod p}$ is an isomorphism; in addition it is $G$-equivariant, since the product morphism of the ring $\phi_*\Oc_U$ is $G$-equivariant. We thus obtain a $G$-equivariant decomposition 
\begin{equation}
\label{eq:G_decomposition}
\phi_*\Oc_U \simeq \bigoplus_{i=0}^{p-1} \Lc^{\otimes i}.
\end{equation}

Let $\alpha \in K_G(X)$. Consider the commutative diagram
\[ \xymatrix{
K_0'(U;G)\ar[rr]^{\eqref{eq:KG-K_G}} && K_G(U)\\ 
K_0'(U/\mup;G) \ar[rr]^{\eqref{eq:KG-K_G}}\ar[u]^{\phi^*}&& K_G(U/\mup) \ar[u]_{\phi^*}
}\]
Since the upper horizontal and left vertical morphisms are surjective, respectively by \rref{lemm:KG-K_G} and \rref{prop:K_free}, it follows that the right vertical morphism is surjective as well. Thus we may find $\beta \in K_G(U/\mup)$ such that $\phi^*(\beta) = \alpha|_U \in K_G(U)$. By the projection formula \cite[IV, (2.12.4)]{sga6} and the decomposition \eqref{eq:G_decomposition}, we have in $K_G(U/\mup)$
\begin{equation}
\label{eq:proj}
\phi_*(\alpha|_U) = \phi_* \circ \phi^*(\beta) = \sum_{i=0}^{p-1} [\Lc]^i \cap \beta.
\end{equation}
There is a polynomial $Q$ with integral coefficients such that we have in $K_0(U/\mup;G)$
\begin{equation}
\label{eq:Q}
\sum_{i=0}^{p-1} [\Lc]^i = (1-[\Lc])^{p-1} +p\cdot Q([\Lc]).
\end{equation}
Since $(1-[\Lc]) \cap K_G(U/\mup)_{(n)} \subset K_G(U/\mup)_{(n-1)}$ for every $n$ (see \cite[VI, Proposition~5.2]{FL-Ri-85} or \cite[X, Th\'eor\`eme 1.3.2]{sga6}), and $\dim U/\mup \leq p-2$, we have $(1-[\Lc])^{p-1} \cap \beta \in K_G(U/\mup)_{(-1)}=0$, and the combination of \eqref{eq:proj} and \eqref{eq:Q} yields 
\[
\phi_*(\alpha|_U)  \in p  \cdot K_G(U/\mup).
\]
By \rref{lemm:loc-K_G} we can find $y \in K_G(X/\mup)$ such that $\varphi_*(\alpha) -py$ restricts to zero in $K_G(U/\mup)$, where $\varphi \colon X \to X/\mup$ denotes the quotient morphism. The morphism $i \colon X^{\mup} \to X/\mup$ is a closed immersion by \cite[(3.2.3.ii)]{isol}, with open complement $U/\mup$. By \rref{lemm:loc-K_G} we may find $z \in K_G(X^{\mup})$ such that $i_*(z)=\varphi_*(\alpha) -py$. To conclude, observe that
\[
\delta(\alpha) = \delta \circ \varphi_*(\alpha) = \delta (\varphi_*(\alpha) -py) = \delta \circ i_*(z) = \delta(z) \in \Fp.\qedhere
\]
\end{proof}

\begin{proposition}
\label{prop:denom_RR}
Let $X$ be a projective variety such that $\dim X <p-1$. Then the Euler characteristic $\chi(X,\Fc)$ of any coherent $\Oc_X$-module $\Fc$ is congruent modulo $p$ to the degree of some zero-cycle supported on $X$. 
\end{proposition}
\begin{proof}
The group $K_0'(X)$ is generated by classes $[\Oc_Z]$ for $Z$ a closed subscheme of $X$ (see \cite[X, Corollaire 1.1.4]{sga6}). By \cite[Proposition~9.1]{firstst}, for such $Z$ the integer $\chi(Z,\Oc_Z)$ is congruent modulo $p$ to the degree of some zero-cycle supported on $Z$, and in particular on $X$.
\end{proof}

\begin{theorem}
\label{th:smalldim}
Let $G$ be a finite algebraic $p$-group acting on a projective variety $X$. Assume that the degree of every closed point of $X^G$ is divisible by $p$. If $\carac k=p$, assume that there are normal subgroups $1=G_0 \subset G_1 \subset \cdots \subset G_n =G$ such that the subgroup $G_i/G_{i-1}$ is central in $G/G_{i-1}$ and isomorphic to $\mu_p$, for $i=1,\cdots,n$.
\begin{enumerate}[label=(\roman*),ref=\roman*]
\item \label{th:smalldim:K}
If $\dim X < p-1$, then the Euler characteristic $\chi(X,\Fc)$ of any $G$-equivariant coherent $\Oc_X$-module is divisible by $p$.

\item \label{th:smalldim:CH} 
If $\dim X < p-1$, then the morphism $\CH_G(X) \to \CH(X) \xrightarrow{\deg} \Fp$ is zero.

\item \label{th:smalldim:Chern} 
If $\dim X < p$ and $X$ is smooth, then any Chern number (\S\ref{sect:Chern}) of $X$ is divisible by $p$.
\end{enumerate}
\end{theorem}
\begin{proof}
Let us first assume that $\carac k \neq p$. In order to prove the theorem, we may replace $k$ with an extension of degree prime to $p$. Applying \rref{lemm:subgroup} recursively and extending scalars, we obtain normal subgroups $1=G_0 \subset G_1 \subset \cdots \subset G_n =G$ such that $G_i/G_{i-1}$ is central in $G/G_{i-1}$ and isomorphic to $(\Zz/p)_k$. Enlarging once more the field $k$, we may assume that it contains a nontrivial $p$-th root of unity, so that $(\Zz/p)_k \simeq \mu_p$. 

It will thus suffice to prove the theorem when  the characteristic of $k$ is arbitrary, assuming that $G$ admits normal subgroups $1=G_0 \subset G_1 \subset \cdots \subset G_n =G$ such that the subgroup $G_i/G_{i-1}$ is central in $G/G_{i-1}$ and isomorphic to $\mu_p$.

Assume that $\dim X <p-1$. We prove by descending induction on $i=n,\cdots,0$ that
\begin{equation}
\label{eq:induction}
\delta(K_{G/G_i}(X^{G_i})) = 0.
\end{equation}
If $i=n$, then $G=G_n$, and $\delta(K_{G/G_n}(X^G))$ is contained in the image of $K_0'(X^G) \xrightarrow{\chi(X,-)} \Zz \to \Fp$. The latter is zero by \rref{prop:denom_RR} and the assumption on $X^G$, so that \eqref{eq:induction} holds for $i=n$. Now \eqref{eq:induction} follows by induction, since for $i \in \{1,\cdots,n\}$ we have
\[
\delta\big(K_{G/G_{i-1}}(X^{G_{i-1}})\big) = \delta\big(K_{G/G_{i-1}}(X^{G_i})\big) = \delta\big(K_{G/G_i}(X^{G_i})\big).
\]
Indeed, the first equality follows by applying \rref{prop:delta_fixed} to the central subgroup $G_i/G_{i-1} \simeq \mup$ of $G/G_{i-1}$ acting on the variety $X^{G_{i-1}}$. The second equality follows from \rref{lemm:reduce_group}, since the diagonalisable central subgroup $G_i/G_{i-1}\simeq \mup$ of $G/G_{i-1}$ acts trivially on the variety $X^{G_i}$.

Taking $i=0$ in \eqref{eq:induction}, we obtain $\delta(K_G(X)) = 0$. Since the morphism $\KG{X}{G} \to K_0'(X)$ factors through $K_G(X)$, we have proved \eqref{th:smalldim:K}.

Using the notation of \S\ref{sect:KG}, the morphism $\CH_G(X) \to \CH(X)$ factors as
\[
\CH_G(X) \to \CH_G(X \times T) \simeq \CH( (X \times T)/G) \to \CH(X \times T) \simeq \CH(X).
\]
The last isomorphism follows from \rref{lemm:pt} (in view of \cite[Proposition~52.9]{EKM}), and the first from \cite[Proposition 8(a)]{EG-Equ} (since $G$ acts freely on $X \times T$). Consider the composite
\[
\Delta \colon \CH((X \times T)/G) \to \CH(X \times T) \simeq \CH(X) \xrightarrow{\deg} \Fp.
\]
Since $\dim T=0$, we have $\Delta(\CH_d( (X \times T)/G)) = 0$ unless $d=0$. Thus in order to prove \eqref{th:smalldim:CH}, it will suffice to consider a zero-dimensional closed subscheme $Z$ of $(X \times T)/G$, and prove that $\Delta[Z]=0$. But for such $Z$, we have $\Delta[Z] = \delta[\Oc_Z]$, which vanishes by \eqref{th:smalldim:K}.

We now prove \eqref{th:smalldim:Chern} by induction on the length $n$ of the central series. The statement is true if $n=0$. Assume that $n \geq 1$ and $\dim X <p$. Let $Y$ be the scheme theoretic closure of $U=X-X^{G_1}$ in $X$, and $Z$ the open complement of $Y$. Then $Y$ is the union of the connected components of $X$ meeting $U$ (as $X$ is smooth), hence is open in $X$. Thus $X$ decomposes $G$-equivariantly as the disjoint union of the open subschemes $Y$ and $Z$. Every Chern number of $X$ is the sum of a Chern number of $Y$ and one of $Z$. Applying the induction hypothesis to the $G/G_1$-action on the variety $Z$ (which is a closed subscheme of $X^{G_1}$ by construction), we may replace $X$ with $Y$, and thus assume that $\dim X^{G_1} < p-1$. Since the tangent bundle of $X$ is $G$-equivariant by \rref{rem:Tan_is_equ}, its Chern classes lie in the image of $\CH_G(X) \to \CH(X)$, hence it will suffice to prove that the composite $\CH_G(X) \to \CH(X) \xrightarrow{\deg} \Fp$ is zero. To do so, we proceed as in the proof of \rref{prop:delta_fixed}. The group $G_1 \simeq \mup$ acts freely on $U$ by \rref{lemm:nofix_free}. Denote by $\varphi\colon X \to X/G_1$ and $\phi \colon U \to U/G_1$ be the quotient morphisms, and let $\alpha \in \CH_G(X)$. Then the pull-back morphism $\phi^* \colon \CH_G(U/G_1) \to \CH_G(U)$ is surjective by \rref{lemm:free_equpb}, so that we may find $\beta \in \CH_G(U/G_1)$ such that $\phi^*(\beta)=\alpha|_U$. By the localisation sequence of \S\ref{sect:CHequ_loc}, we may find $\beta' \in \CH_G(X/G_1)$ such that $\beta'|_{U/G_1} = \beta$. Since the morphism $\phi$ is flat and finite of degree $p$ by \rref{prop:G_torsor}, the endomorphism $\phi_* \circ \phi^*$ of $\CH_G(U/G_1)$ is multiplication by $p$, and it follows that $\varphi_*(\alpha) -p \beta'$ restricts to zero in $\CH_G(U/G_1)$. The morphism $i \colon X^{G_1} \to X/G_1$ is a closed immersion by \cite[(3.2.3.ii)]{isol}, with open complement $U/G_1$. Using the localisation sequence of \S\ref{sect:CHequ_loc}, we see that $\varphi_*(\alpha) = p \beta' + i_*(\gamma)$ for some $\gamma \in \CH_G(X^{G_1})$. Since $\dim X^{G_1} <p-1$, we may conclude using \eqref{th:smalldim:CH}.
\end{proof}

\begin{lemma}
\label{lemm:subgroup}
Let $G$ be an \'etale $p$-group satisfying $G \neq 1$. Then, after an extension of degree prime to $p$ of the base field, the algebraic group $G$ possesses a central subgroup isomorphic to the constant algebraic group $(\Zz/p)_k$.
\end{lemma}
\begin{proof}
Let $C$ be the centre of $G$ and $\overline{k}$ an algebraic closure of $k$. Then the ordinary group $C(\overline{k})$ is the centre of the nontrivial ordinary $p$-group $G(\overline{k})$, hence is nontrivial. Thus the algebraic group $C$ is nontrivial (and is again an \'etale $p$-group). The open complement of the unit in $C$ is a variety which is reduced and finite of degree prime to $p$, hence possesses a closed point of degree prime to $p$. Extending scalars, we may assume that the ordinary $p$-group $C(k)$ is nontrivial, hence contains a subgroup isomorphic to $\Zz/p$, which implies that the algebraic group $C$ contains a subgroup isomorphic to the constant algebraic group $(\Zz/p)_k$.
\end{proof}

\subsection{Characteristic \texorpdfstring{$p$}{p}}
In the next statement it is necessary to assume that $k$ is perfect, as shown by \rref{ex:SB} and \rref{ex:zero_cycle}. 

\begin{theorem}
\label{th:K:carac}
Let $k$ be a perfect field of characteristic $p>0$. Let $G$ be either a constant algebraic $p$-group, or an extension of one by $\alpha_p$. Let $X$ be a projective variety with a $G$-action such that the degree of every closed point of $X^G$ is divisible by $p$. Then the Euler characteristic $\chi(X,\Fc)$ of any $G$-equivariant coherent $\Oc_X$-module $\Fc$ is divisible by $p$.
\end{theorem}
\begin{proof}
We proceed by induction on the order of $G$. When the group $G$ is trivial, the statement follows from \rref{lemm:Rost} below. Note that referring to that result is unnecessary when $k$ is algebraically closed (in particular while proving \rref{th:constant_CH}): since the variety $X$ possesses no closed point of degree prime to $p$, it must be empty, so that $\chi(X,\Fc)=0$.

If the group $G$ is nontrivial, it admits a normal subgroup $C$ isomorphic to $\alpha_p$ or $(\Zz/p)_k$, and such that $G/C$ is a constant $p$-group. Let $Y$ be the scheme theoretic image of the composite morphism $X^C \to X \to X/C$, and denote by $\varphi \colon X \to X/C$ the quotient morphism. Let $U=\varphi^{-1}(X/C - Y)$, and $\psi \colon U \to U/C = X/C - Y$ the restriction of $\varphi$. Since $U^C = \varnothing$, it follows from \rref{lemm:nofix_free} that the group $C$ acts freely on $U$. Thus by \rref{prop:K_free}, there is an element $\beta \in \KG{U/C}{G}$ such that $[\Fc]|_U  = \psi^*(\beta)$ in $\KG{U}{G}$. By \rref{lemm:quotient_unipotent}, we have $\psi_*([\Fc]|_U) = p\beta \in K_0'(U/C;G)$. By the localisation sequence \cite[Theorem~2.7]{Thomason-group}, it follows that $\varphi_*[\Fc] \in K_0'(X/C;G)$ is modulo $p$ the image of an element of $K_0'(Y;G)$, and thus by \rref{lemm:trivial_unipotent} of $K_0'(Y;G/C)$. Since $\chi(X,\Fc) = \chi(X/C,\varphi_*\Fc)$, it will suffice to prove that $\chi(Y,\Ec)$ is divisible by $p$ for any $G/C$-equivariant coherent $\Oc_Y$-module $\Ec$. To do so, it will suffice by induction to prove that the degree of any closed point $y\in Y^{G/C}$ is divisible by $p$. The field $k(y)$ is perfect, being a finite extension of the perfect field $k$. The $G/C$-equivariant morphism $X^C \to Y$ is surjective, and radicial by \rref{lemm:quotient_radicial}. Therefore by \rref{lemm:smooth_radicial} (applied to the smooth group $G/C$) we may find a closed point of $(X^C)^{G/C} = X^G$ (mapping to $y$) with residue field isomorphic to $k(y)$, so that the degree of $y$ must be divisible by $p$, by our assumption on $X^G$.
\end{proof}

The next lemma is due to Rost \cite[Corollary~1]{Ros-On-08}.
\begin{lemma}
\label{lemm:Rost}
Let $k$ be a perfect field of characteristic $p>0$. Let $X$ be a projective variety, and $n$ a positive integer. If the degree of every closed point of $X$ is divisible by $p^n$, then so is the Euler characteristic $\chi(X,\Fc)$ of any coherent $\Oc_X$-module $\Fc$.
\end{lemma}
\begin{proof}
We proceed by induction on $d =\dim X$. Since the group $K_0'(X)$ is generated by classes $[\Oc_Z]$ for $Z$ an integral closed subscheme of $X$ (see \cite[X, Corollaire 1.1.4]{sga6}), we may assume that $X$ is integral and that $\Fc = \Oc_X$. The case $d=0$ is clear, since $\chi(X,\Oc_X) = [k(X):k]$. Assume that $d>0$. Denote by $X^{(p)} = X \times_{\Spec k} \Spec k$ the base-change of $X$ under the absolute Frobenius of $\Spec k$. Since $k$ is perfect, the first projection $\pi \colon X^{(p)} \to X$ is an $\Fp$-isomorphism. In particular $X^{(p)}$ is integral; let $\eta$ be its generic point. The relative Frobenius $F \colon X \to X^{(p)}$ is a finite $k$-morphism. By \cite[Corollary~3.2.27]{Liu} the $k(X^{(p)})$-vector space $(F_*\Oc_X)_\eta = k(X)$ has dimension $p^d$, hence there is a nonempty open subscheme $V$ of $X^{(p)}$ such that the $\Oc_V$-module $(F_*\Oc_X)|_V$ is free of rank $p^d$. By the localisation sequence for $K_0'$ (see \cite[Proposition 7.3.2]{Qui-72}), this implies that $[F_*\Oc_X] - p^d [\Oc_{X^{(p)}}] = i_*\varepsilon \in K_0'(X^{(p)})$ for some $\varepsilon \in K_0'(Y)$, where $i\colon Y \to X^{(p)}$ is the immersion of the reduced closed complement of $V$. The composite $\pi \circ i \colon Y \to X^{(p)} \to X$ maps closed points to closed points (being a closed immersion), multiplying their degrees by $[k:k^p]=1$. Thus the degree of every closed point of $Y$ is a multiple of $p^n$, hence $p^n$ divides $\chi(Y,\varepsilon)$ by induction. It follows from \eqref{def:Euler} and \cite[Corollary 5.2.27]{Liu} that $\chi(X^{(p)},\Oc_{X^{(p)}}) = \chi(X,\Oc_X)$. We have $\chi(X^{(p)},F_*\Oc_X)=\chi(X,\Oc_X)$, because $F$ is a finite $k$-morphism. Thus 
\[
\chi(Y,\varepsilon) = \chi(X^{(p)},i_*\varepsilon) = \chi(X^{(p)},F_*\Oc_X) - p^d \chi(X^{(p)},\Oc_{X^{(p)}}) = (1-p^d)\chi(X,\Oc_X).
\]
Since $d>0$ and $p^n$ divides $\chi(Y,\varepsilon)$, we conclude that $p^n$ must divide $\chi(X,\Oc_X)$.
\end{proof}

\begin{lemma}
\label{lemm:trivial_unipotent}
Let $G$ be an algebraic group acting on a variety $X$. Let $H$ be a normal unipotent subgroup of $G$ acting trivially on $X$. Then the natural morphism $\KG{X}{G/H} \to \KG{X}{G}$ 
is surjective.
\end{lemma}
\begin{proof}
Let $\Fc$ be a $G$-equivariant coherent $\Oc_X$-module. We construct inductively $G$-equivariant coherent $\Oc_X$-modules $\Fc_i$ for $i \in \mathbb{N}$ as follows. Let $\Fc_0=\Fc$. Assuming $\Fc_i$ constructed, we consider its $G$-equivariant coherent $\Oc_X$-submodule $(\Fc_i)^H$ (see \rref{def:invariants}), and define $\Fc_{i+1} = \Fc_i/(\Fc_i)^H$. As $\Fc$ is coherent and $X$ noetherian, there is an integer $n$ such that the morphism $\Fc_i \to \Fc_{i+1}$ is an isomorphism for all $i \geq n$. Thus we have $(\Fc_i)^H = 0$ for $i \geq n$. We claim that this implies that $\Fc_i=0$ for such $i$. Indeed, let $U$ be an open subscheme of $X$ such that $\Fc_i(U) \neq 0$. Then $\Fc_i(U)$ is a nonzero $H$-representation over $k$, and since $H$ is unipotent, it follows that $\Fc_i(U)^H\neq 0$. But $(\Fc_i)^H(U) = \Fc_i(U)^H$ by \rref{rem:invariant_sections}, a contradiction.

Using the relations $[\Fc_i] = [\Fc_{i+1}] + [(\Fc_i)^H]$ for $i=0,\cdots,n-1$ and $[\Fc_n]=0$ in $\KG{X}{G}$, we see that $[\Fc]=[\Fc_0] \in \KG{X}{G}$ is the image of an element of $\KG{X}{G/H}$, namely $[(\Fc_0)^H] +\cdots + [(\Fc_{n-1})^H]$.
\end{proof}

\begin{lemma}
\label{lemm:quotient_unipotent}
Let $G$ be a unipotent algebraic group acting on a variety $X$. Let $H$ be a finite normal subgroup of $G$ acting freely on $X$. Let $\varphi \colon X \to X/H$ be the quotient morphism. Then the endomorphism $\varphi_*\circ \varphi^*$ of $\KG{X/H}{G}$ is multiplication by the order of $H$.
\end{lemma}
\begin{proof}
Let $R$ be the coordinate ring of $H$, and $d=\dim_k R$ the order of $H$. We let $G$ act on $H$ by $(g,h) \mapsto ghg^{-1}$. The commutative square of $G$-equivariant morphisms
\[ \xymatrix{
H \times X\ar[r]^{\alpha} \ar[d]_{\pi} & X \ar[d]^\varphi \\ 
X \ar[r]^\varphi & X/H
}\]
where $\alpha,\pi$ denote respectively the action and projection, is cartesian by \rref{prop:G_torsor}. For any $G$-equivariant coherent $\Oc_{X/H}$-module $\Fc$, the $G$-equivariant morphism of $\Oc_X$-modules 
\[
(\varphi^*\Fc) \otimes_k R = \pi_*\pi^* \varphi^*\Fc = \pi_*\alpha^* \varphi^*\Fc \to \varphi^* \varphi_*\varphi^*\Fc
\]
is an isomorphism by \cite[(1.5.2)]{ega-2}. The forgetful morphism $\KG{\Spec k}{G} \to K_0'(\Spec k)$ is injective, being a retraction of a surjective morphism by \rref{lemm:trivial_unipotent}. It follows that $[R] = d \in \KG{\Spec k}{G}$. Thus
\begin{equation}
\label{eq:phiphiphi}
\varphi^* \circ \varphi_* \circ \varphi^* = d \cdot \varphi^* \colon \KG{X/H}{G} \to \KG{X}{G}.
\end{equation}
The composite $\KG{X/H}{G/H} \to \KG{X/H}{G} \xrightarrow{\varphi^*} \KG{X}{G}$ is bijective by \rref{prop:invariant_free}, and the first morphism is surjective by \rref{lemm:trivial_unipotent}. Thus the second morphism $\varphi^*$ is bijective, and the lemma follows from \eqref{eq:phiphiphi}
\end{proof}

\begin{lemma}
\label{lemm:smooth_radicial}
Let $G$ be a smooth algebraic group and $Y \to X$ a $G$-equivariant surjective radicial morphism. Then the map $Y^G(K) \to X^G(K)$ is bijective for any perfect field extension $K/k$ with trivial $G$-action.
\end{lemma}
\begin{proof}
Let $F$ be the fiber of $Y \to X$ over a point in $X^G(K)$, and $F_{red}$ the underlying reduced closed subscheme. The scheme $F$, and thus also $F_{red}$, is nonempty and radicial over $\Spec K$. Since $K$ is perfect, the morphism $F_{red}\to \Spec K$ is an isomorphism. As the group $G$ is reduced, the closed subscheme $F_{red}$ of $F$ is $G$-invariant. It follows that the isomorphism $F_{red}\to \Spec K$ is $G$-equivariant. Its inverse provides a $G$-equivariant morphism $\Spec K \to F_{red} \to F \to Y$, and therefore the required $K$-point of $Y^G$.
\end{proof}

We now provide two examples showing that \rref{th:K:carac} does not generalise to arbitrary algebraic $p$-groups.

\begin{example}
\label{ex:aa}
Let $k$ be a field of characteristic two. The subgroups of $\GL_2$
\[
\begin{pmatrix}
1&\alpha_2\\
0&1
\end{pmatrix}
\text{ and }
\begin{pmatrix}
1&0\\
\alpha_2&1
\end{pmatrix}
\]
induce two actions of $\alpha_2$ on $\Pp^1$. For a field extension $K/k$ with trivial $G$-action, the sets $(\Pp^1)^{\alpha_2}(K)$ are respectively $\{[1:0]\}$ and $\{[0:1]\}$. Since the two actions commute with one another, this gives an action of $\alpha_2 \times \alpha_2$ on $\Pp^1$ without fixed point, and $\chi(\Pp^1,\Oc_{\Pp^1}) =1$ is odd.
\end{example} 

\begin{example}
\label{ex:amu}
Let $k$ be a field of characteristic two. Sending $a \in \alpha_2(R)$, for a commutative $k$-algebra $R$, to the matrix
\[
\begin{pmatrix}
1&a\\
a&1
\end{pmatrix} \in \GL_2(R).
\]
defines a group morphism $\alpha_2 \to \PGL_2$. If $K/k$ is a field extension with trivial $G$-action, the only fixed $K$-point for the induced action of $\alpha_2$ on $\Pp^1$ is $[1:1]$.

We define a group morphism $\mu_2 \to \GL_2$ by sending $\xi \in \mu_2(R)$, for a commutative $k$-algebra $R$, to the matrix
\[
\begin{pmatrix}
1&0\\
0&\xi
\end{pmatrix} \in \GL_2(R).
\]
If $K/k$ is a field extension with trivial $G$-action, the only fixed $K$-points for the induced action of $\mu_2$ on $\Pp^1$ are $[0:1]$ and $[1:0]$.

Letting $\mu_2$ act on $\alpha_2$ via the restriction of the natural action of $\Gm$ on $\Ga$, we obtain an action of $\alpha_2 \rtimes \mu_2$ on $\Pp^1$ without fixed point, and $\chi(\Pp^1,\Oc_{\Pp^1}) =1$ is odd.
\end{example}

\begin{remark}
It would be interesting to determine for which algebraic $p$-groups $G$ in characteristic $p$ does the arithmetic genus detect fixed points (in the sense of \ref{req:detect} in the introduction). We have proved that it does when $G\in \{\alpha_p \times \Zz/p, \Zz/p \times \Zz/p, \Zz/p^2,\mu_{p^2}\}$ (see \rref{th:K:carac} and \rref{th:cyclic}), but not when $G \in \{\alpha_p \times \alpha_p, \mup \times \Zz/p, \alpha_p \rtimes \mu_p\}$ (see \rref{ex:aa}, \rref{ex:mup_zp}, \rref{ex:amu})
\end{remark}

\section{Browder's theorem}
\label{sect:browder}
\numberwithin{theorem}{section}
\numberwithin{lemma}{section}
\numberwithin{proposition}{section}
\numberwithin{corollary}{section}
\numberwithin{example}{section}
\numberwithin{notation}{section}
\numberwithin{definition}{section}
\numberwithin{remark}{section}
\counterwithin{equation}{section}
\numberwithin{equation}{section}
\renewcommand{\theequation}{\thesection.{\alph{equation}}}

The next statement may be thought of as a relative version of \rref{th:eq_deg} and \dref{th:smalldim}{th:smalldim:CH}. It implies an algebraic version of a theorem of Browder in topology (see \S\ref{sect:Browder}).

\begin{theorem}
\label{th:Browder}
Let $G$ be an algebraic $p$-group and $f\colon Y \to X$ a projective $G$-equivariant morphism. Assume that $X$ is smooth, and that one of the following conditions holds:
\begin{enumerate}[label=(\roman*),ref=\roman*]
\item \label{th:Browder:trig} The group $G$ is trigonalisable.

\item \label{th:Browder:smalldim} Every fiber of $f$ has dimension $< p-1$, and $\carac k \neq p$.
\end{enumerate}
Assume the image of $f_* \colon \CH_G(Y)/p \to \CH_G(X)/p$ contains the class $[X]$. Then the morphism $(f^G)_* \colon \CH(Y^G)/p \to \CH(X^G)/p$ is surjective.
\end{theorem}
\begin{proof}
If $T \to X$ is a morphism, resp.\ $x \in X$ is a point, we will denote by $f_T \colon Y_T \to T$, resp.\ $f_x \colon Y_x \to \Spec k(x)$, the base-change of $f$.

We first prove that for any point $x \in X^G$ the morphism
\begin{equation}
\label{eq:deg_fiber}
\CH_G(Y_x) \to \CH(Y_x) \xrightarrow{\deg}  \CH(\Spec k(x))/p = \Fp
\end{equation}
is surjective. Let $\alpha' \in \CH_G(Y)$ be such that $f_*(\alpha') =[X] \in \CH_G(X)/p$, and let $\alpha \in \CH(Y)$ be the image of $\alpha'$. Let $R$ be the reduced closure of $x$ in $X$. Since $f_*(\alpha) = [X] \in \CH(X)/p$, for any $\gamma \in \CH(R)/p$ we have by the projection formula \cite[8.1.1.c]{Ful-In-98}
\begin{equation}
\label{eq:g_*}
(f_R)_*(\alpha \cdot_f \gamma) = (f_*\alpha) \cdot_{\id_X} \gamma = \gamma \in \CH(R)/p.
\end{equation}
Let $\gamma' \in \CH_G(R)$ be a preimage of $\gamma$ (the morphism $\CH_G(R) \to \CH(R)$ is surjective because $G$ acts trivially on $R$). Then the element $\alpha' \cdot_f \gamma' \in \CH_G(Y_R)$ constructed in \rref{sect:equ_refined} is a preimage of $\alpha \cdot_f \gamma \in \CH(Y_R)$. In view of \eqref{eq:g_*}, this proves that the composite $\CH_G(Y_R) \to \CH(Y_R) \to \CH(R)/p$ is surjective. Let $r=\dim R$. We may find a $G$-invariant open subscheme $W$ of a finite dimensional $G$-representation over $k$ whose complement has codimension $>\dim Y_x =\dim Y_R - r$, such that $G$ acts freely on $W$. Let $w= \dim W$. It follows from \cite[Propositions 49.18 and 49.20]{EKM} that we have a commutative diagram (where $W_{k(x)} = W \times_k \Spec k(x)$)
\[ \xymatrix{
\CH_{w+r}((Y_R \times_k W)/G) \ar[r] \ar[d] & \CH_{w+r}(Y_R \times_k W) \ar[r] \ar[d] & \CH_{w+r}(R \times W)\ar[d] \\ 
\CH_w((Y_x \times_{k(x)} W_{k(x)})/G)  \ar[r] & \CH_w(Y_x \times_{k(x)} W_{k(x)}) \ar[r] & \CH_w(W_{k(x)})
}\]
The right vertical morphism is surjective by the localisation sequence \cite[Proposition~1.8]{Ful-In-98} and \cite[Proposition 52.9]{EKM}. The upper horizontal composite may be identified with the degree $r$ component of the composite $\CH_G(Y_R) \to \CH(Y_R) \to \CH(R)$. The lower horizontal composite may be identified with the degree $0$ component of the composite $\CH_G(Y_x) \to \CH(Y_x) \to \CH(\Spec k(x))$. Since the former is surjective modulo $p$, so is the latter, proving that \eqref{eq:deg_fiber} is surjective.

We now prove that for any closed subscheme $S$ of $X^G$, the morphism $\CH( (Y_S)^G)/p \to \CH(S)/p$ is surjective. The theorem will then follow from the case $S=X^G$. We proceed by noetherian induction, the case $S=\varnothing$ being clear. Let $s$ be a generic point of $S$. For any closed subscheme $Z$ of $S$ not containing $s$, we have a commutative diagram with exact rows, where $U=S-Z$ (see \cite[Propositions 1.7 and 1.8]{Ful-In-98})
\[ \xymatrix{
\CH((Y_Z)^G)/p\ar[r] \ar[d] & \CH((Y_S)^G)/p \ar[d] \ar[r] & \CH((Y_U)^G)/p \ar[r] \ar[d]&0 \\ 
\CH(Z)/p \ar[r] & \CH(S)/p  \ar[r] & \CH(U)/p \ar[r] &0
}\]
By induction the left vertical morphism is surjective, hence so is its colimit over the closed subschemes $Z$ not containing $s$. In view of \rref{th:eq_deg} in case \eqref{th:Browder:trig} and \dref{th:smalldim}{th:smalldim:CH} in case \eqref{th:Browder:smalldim}, the surjectivity of \eqref{eq:deg_fiber} for $x=s$ implies that $\CH((Y_s)^G)/p \to \CH(\Spec k(s))/p = \Fp$ is also surjective. The latter morphism may be identified with the colimit of the right vertical morphism in the above diagram. We conclude that the middle vertical morphism is surjective, as required.
\end{proof}

\begin{remark}
\label{rem:Browder}
The conclusion of the theorem implies the following:
\begin{enumerate}[label=(\roman*),ref=\roman*]
\item \label{rem:Browder:surj} $f^G$ is surjective, and in particular $\dim Y^G \geq \dim X^G$.

\item If $X^G$ is projective and possesses a closed point of degree $n$, then $Y^G$ possesses a closed point of some degree $m$ whose $p$-adic valuation satisfies $v_p(m) \leq v_p(n)$.

\item \label{rem:Browder:smooth} If $X^G$ is smooth (for instance if $p\neq \carac k$ by \cite[Proposition A.8.10 (2)]{CGP}), then the $\Zz_{(p)}$-module $\CH(X^G) \otimes \Zz_{(p)}$ is a direct summand of $\CH(Y^G)\otimes \Zz_{(p)}$. More generally $H(X^G)\otimes \Zz_{(p)}$ is a direct summand of $H(Y^G)\otimes \Zz_{(p)}$ when $H$ is a cohomology theory such as motivic cohomology, $K$-theory or algebraic cobordism (see \cite[Conditions~3.1, Lemma~4.1]{invariants} for precise conditions). 
\end{enumerate}
\end{remark}

\end{document}